\documentclass[10pt,a4paper,american]{amsart}
\usepackage{babel}
\usepackage{amsmath,amsthm,bm}
\usepackage{amsbsy}
\usepackage{amstext}
\usepackage{amsfonts}
\usepackage{floatflt}
\usepackage[pdftex]{graphicx}
\usepackage{mathcomp}
\usepackage{mathtools}
\usepackage[dvipsnames]{xcolor}
\usepackage{tikz,pgfplots}
\usepackage{dsfont}
\usepackage{latexsym}
\usepackage{amssymb}
\usepackage{stmaryrd}
\usepackage{bm}
\usepackage{enumerate}
\usepackage{esint}	
\usepackage[colorlinks,pdfpagelabels,pdfstartview = FitH,bookmarksopen = true,bookmarksnumbered = true,linkcolor = blue,plainpages = false,hypertexnames = false,citecolor = red,pagebackref=false]{hyperref}
\usepackage{cite}

\allowdisplaybreaks
\makeatletter
\newsavebox{\@brx}
\newcommand{\llangle}[1][]{\savebox{\@brx}{\(\m@th{#1\langle}\)}%
  \mathopen{\copy\@brx\kern-0.5\wd\@brx\usebox{\@brx}}}
\newcommand{\rrangle}[1][]{\savebox{\@brx}{\(\m@th{#1\rangle}\)}%
  \mathclose{\copy\@brx\kern-0.5\wd\@brx\usebox{\@brx}}}
\makeatother

\newtheorem{theorem}{Theorem}
\newtheorem{lemma}[theorem]{Lemma}

\theoremstyle{definition}
\newtheorem{definition}[theorem]{Definition}

\theoremstyle{remark}

\numberwithin{theorem}{section}
\numberwithin{equation}{section}

\DeclareMathOperator{\Div}{div}

\newcommand{\N}{\ensuremath{\mathbb{N}}}
\newcommand{\R}{\ensuremath{\mathbb{R}}}
\newcommand{\mint}{- \mskip-19,5mu \int}

\renewcommand{\u}{\boldsymbol{u}}

\newcommand{\lbbracket}{\boldsymbol{(}}
\newcommand{\rbbracket}{\boldsymbol{)}}

\newcommand{\dx}{\mathrm{d}x}
\newcommand{\dt}{\mathrm{d}t}

\makeatletter
\@namedef{subjclassname@2020}{%
  \textup{2020} Mathematics Subject Classification}
\makeatother

\newcommand{\power}[2]{\bm{#1^{\mbox{\unboldmath{\scriptsize$#2$}}}}}

\newcommand{\babs}[1]{\big|#1\big|}

\def\Xint#1{\mathchoice
    {\XXint\displaystyle\textstyle{#1}}%
    {\XXint\textstyle\scriptstyle{#1}}%
    {\XXint\scriptstyle\scriptscriptstyle{#1}}%
    {\XXint\scriptscriptstyle\scriptscriptstyle{#1}}%
    \!\int}
\def\XXint#1#2#3{\setbox0=\hbox{$#1{#2#3}{\int}$}
    \vcenter{\hbox{$#2#3$}}\kern-0.5\wd0}
\def\bint{\Xint-}
\def\dashint{\Xint{\raise4pt\hbox to7pt{\hrulefill}}}

\def\Xiint#1{\mathchoice
    {\XXiint\displaystyle\textstyle{#1}}%
    {\XXiint\textstyle\scriptstyle{#1}}%
    {\XXiint\scriptstyle\scriptscriptstyle{#1}}%
    {\XXiint\scriptscriptstyle\scriptscriptstyle{#1}}%
    \!\iint}
\def\XXiint#1#2#3{\setbox0=\hbox{$#1{#2#3}{\iint}$}
    \vcenter{\hbox{$#2#3$}}\kern-0.5\wd0}
\def\biint{\Xiint{-\!-}}

\renewcommand{\epsilon}{\varepsilon}
\newcommand{\eps}{\varepsilon}
\renewcommand{\rho}{\varrho}

\renewcommand{\epsilon}{\varepsilon}
\renewcommand{\rho}{\varrho}
\renewcommand{\d}{\:\! \mathrm{d}}

\renewcommand{\u}{\boldsymbol{u}}

\renewcommand{\a}{\boldsymbol{a}}

\newcommand{\A}{\mathbf{A}}

\DeclareMathOperator{\loc}{loc}

\numberwithin{equation}{section}
\allowdisplaybreaks


\begin{document}
\renewcommand{\refname}{References} 
\renewcommand{\abstractname}{Abstract} 
\title[Higher integrability for singular doubly nonlinear systems]{Higher integrability for singular doubly nonlinear systems}
\author[K.~Moring]{Kristian Moring}
\address{Kristian Moring\\
	Fakult\"at f\"ur Mathematik, Universit\"at Duisburg-Essen\\
	Thea-Leymann-Str.~9, 45127 Essen, Germany}
\email{kristian.moring@uni-due.de}

\author[L.~Schätzler]{Leah Schätzler}
\address{Leah Schätzler\\
	Fachbereich Mathematik, Paris-Lodron-Universität Salzburg\\
	Hellbrunner Str.~34, 5020 Salzburg, Austria}
\email{leahanna.schaetzler@plus.ac.at}

\author[C.~Scheven]{Christoph Scheven}
\address{Christoph Scheven\\
	Fakult\"at f\"ur Mathematik, Universit\"at Duisburg-Essen\\
	Thea-Leymann-Str.~9, 45127 Essen, Germany}
\email{christoph.scheven@uni-due.de}

\date{}%\today}

\begin{abstract}
We prove a local higher integrability result for the spatial gradient of weak solutions to doubly nonlinear parabolic systems whose prototype is
\begin{equation*}
\partial_t \left(|u|^{q-1}u \right) -\Div  \left( |Du|^{p-2} Du \right) = \Div \left( |F|^{p-2} F \right) \quad \text{ in } \Omega_T := \Omega \times (0,T)
\end{equation*}
with parameters $p>1$ and $q>0$ and $\Omega\subset\R^n$.
In this paper, we are concerned with the ranges $q>1$ and $p>\frac{n(q+1)}{n+q+1}$.
A key ingredient in the proof is an intrinsic geometry that takes both the solution $u$ and its spatial gradient $Du$ into account.
\end{abstract}
\makeatother

\maketitle

\section{Introduction}
Let $\Omega \subset \R^n$, $n \geq 2$, be an open set and $0<T<\infty$.
By $\Omega_T := \Omega \times (0,T)$ we denote the space-time cylinder in $\R^{n+1}$.
In this paper we investigate doubly nonlinear systems of the form
\begin{equation}
	\partial_t \left(|u|^{q-1}u \right) -\Div  \left( |Du|^{p-2} Du \right) = \Div \left( |F|^{p-2} F \right)
	\quad \text{ in } \Omega_T,
	\label{eq:prototype}
\end{equation}
where $q > 0$ and $p >1$.
Here, the solution is a map $u \colon \Omega_T \to \R^N$ for some $N \in \N$.
Applications include the description of filtration processes, non-Newtonian fluids, glaciers, shallow water flows and friction-dominated flow in a gas network, see \cite{Alonso-etal,Bamberger-etal,Kalashnikov,Leugering-Mophou,Mahaffy,Vazquez} and the references therein.
Note that for $q=1$ \eqref{eq:prototype} reduces to the parabolic $p$-Laplace system, while for $p=2$ it is the porous medium system (also called fast diffusion system in the singular case $q>1$).
Further, the homogeneous equation with $p=q+1$ is often called Trudinger's equation in the literature.
This special case divides the parameter range into two parts where solutions to \eqref{eq:prototype} behave differently.
In the slow diffusion case $p>q+1$, information propagates with finite speed and solutions may have compact support whereas in the fast diffusion case $p<q+1$ the speed of propagation is infinite and extinction in finite time is possible.
Further, \eqref{eq:prototype} becomes singular as $u \to 0$ and $Du \to 0$ if $q>1$ and $1<p<2$ respectively, and degenerates as $u \to 0$ and $Du \to 0$ if $0<q<1$ and $p>2$ respectively.
In this paper, we are interested in the singular range $q>1$ with $p > \frac{n(q+1)}{n+q+1}$.
For the precise range that is covered by our main result, see Figure \ref{figure}.
Moreover, we consider general systems
\begin{equation}\label{eq:dne}
\partial_t \left(|u|^{q-1}u \right) -\Div  \A( x,t,u, Du )= \Div \left( |F|^{p-2} F \right) \quad \text{ in } \Omega_T,
\end{equation} 
where $\A \colon\Omega_T\times\R^N\times\R^{Nn}\to\R^{Nn}$ is a Carath\'eodory
function satisfying
\begin{align}
\label{assumption:A}
\begin{aligned}
\left\{
\begin{array}{c}
\A(x,t,u,\xi)\cdot \xi \geq C_o |\xi|^p, \\[5pt]
|\A(x,t,u, \xi)| \leq C_1|\xi|^{p-1}
\end{array}
\right.
\end{aligned}
\end{align}
with positive constants $0< C_o \leq C_1 < \infty$ for a.e.\ $(x,t)\in \Omega_T$ and any $(u,\zeta)\in \R^n \times \R^{Nn}$.
Local weak solutions to \eqref{eq:dne} are given by the following definition.
In particular, the spatial gradient $Du$ lies in the Lebesgue space $L^p(\Omega_T,\R^{N \times n})$, whose integrability exponent corresponds to the structure conditions \eqref{assumption:A} on $\A$.

\begin{definition} \label{def:weak-sol}
Suppose that the vector field $\A \colon \Omega_T \times \R^N \times \R^{Nn} \to \R^{Nn}$ satisfies~\eqref{assumption:A} and $F \in L^p_{\loc}(\Omega_T,\R^{Nn})$. We identify a measurable map $u \colon \Omega_T \to \R^N$ in the class
$$
u \in C\big((0,T);L_{\mathrm{loc}}^{q+1}(\Omega,\R^N)\big) \cap L^p_{\mathrm{loc}}\big(0,T; W^{1,p}_{\mathrm{loc}}(\Omega,\R^N)\big)
$$
as a weak solution to~\eqref{eq:dne} if and only if
$$
\iint_{\Omega_T}  |u|^{q-1}u\cdot \partial_t \varphi - \A (x,t,u,Du) \cdot D \varphi  \, \d x \d t = \iint_{\Omega_T} |F|^{p-2}F \cdot D \varphi \, \d x \d t
$$
for every $\varphi \in C_0^\infty(\Omega_T, \R^N)$.
\end{definition}

Our main result is that the spatial gradient $Du$ of a weak solution to \eqref{eq:dne} is locally integrable to a higher exponent than assumed a priori, provided that $F$ is locally integrable to some exponent $\sigma>p$.
The precise result is the following.

\begin{theorem} \label{thm:main}
Let $1 < q < \max \big\{\frac{n+2}{n-2}, \frac{2p}{n}+1 \big\}$, $p > \frac{n(q+1)}{n+q+1}$, $\sigma > p$ and $F \in L^\sigma_{\loc}(\Omega_T; \R^{Nn})$. Then, there exists $\eps_o = \eps_o(n,p,q,C_o,C_1) \in (0,1]$ such that whenever $u$ is a weak solution to~\eqref{eq:dne} in the sense of Definition~\ref{def:weak-sol}, there holds
$$
Du \in L^{p(1+\eps_1)}_{\loc} (\Omega_T; \R^{Nn}),
$$
in which $\eps_1 = \min \big\{ \eps_o, \frac{\sigma}{p} - 1 \big\}$. Furthermore, there exists $c = c(n,p,q,C_o,C_1) \geq 1$ such that for every $\eps \in (0,\eps_1]$ and $Q_\rho = B_{\rho}(x_o) \times (t_o - \rho^{q+1}, t_o + \rho^{q+1}) \Subset \Omega_T$ the estimate
\begin{align*}
	\biint_{Q_{\frac12 \rho}} |Du|^{p(1+\eps)} \, \d x \d t
	&\leq
	c \left( 1 + \biint_{Q_\rho} \frac{|u|^{p^\sharp}}{\rho^{p^\sharp}} + |F|^p \, \d x \d t \right)^{\eps d}  \biint_{Q_\rho} |Du|^p \, \d x \d t \\
	&\phantom{=}+ c \biint_{Q_\rho} |F|^{p(1+\eps)} \, \d x \d t
\end{align*}
holds true, where $p^\sharp = \max \{p ,q+1 \}$ and
\begin{equation}\label{def-d}
 d= \begin{cases} 
      \frac{p}{q+1} &\text{if } p \geq q+1, \\[5pt]
      \frac{p(q+1)}{p(q+1) +n(p-q-1)} &\text{if } \frac{n(q+1)}{n+q+1} < p < q+1 .
   \end{cases}
 \end{equation}
\end{theorem}

At this stage, some remarks on the history of the problem are in order.
The study of higher integrability was started by Elcrat and Meyers \cite{Meyers-Elcrat}, who gave a result for nonlinear elliptic systems.
Key ingredients of their proof are a Caccioppoli type inequality and the resulting reverse Hölder inequality, and a version of Gehring's lemma.
The latter was originally used in the context of higher integrability for the Jacobian of quasi-conformal mappings in \cite{Gehring}.
For more information, we refer to the monographs \cite[Chapter 5, Theorem 1.2]{Giaquinta:book} and \cite[Theorem 6.7]{Giusti:book}.
The first higher integrability result for parabolic systems is due to Giaquinta and Struwe \cite{Giaquinta-Struwe}, who were able to treat systems of quadratic growth.
However, their technique does not apply to systems of parabolic $p$-Laplace type with general $p\neq 2$.
For $p>\frac{2n}{n+2}$, the breakthrough was achieved by Kinnunen and Lewis \cite{Kinnunen-Lewis:1} (see also \cite{Kinnunen-Lewis:very-weak}), whose key idea was to use a suitable intrinsic geometry.
More precisely, they considered cylinders of the form $Q_{\varrho,\lambda^{2-p}\varrho^2} := B_\varrho(x_o) \times (t_o-\lambda^{2-p}\varrho^2, t_o+\lambda^{2-p}\varrho^2)$, where the length of the cylinder depends on the integral average of $|Du|^p$,
$$
	\lambda^p \approx \biint_{Q_{\varrho,\lambda^{2-p}\varrho^2}} |Du|^p \,\dx\dt.
$$
The concept of intrinsic cylinders has originally been introduced by DiBenedetto and Friedman \cite{DbF2} in connection with Hölder continuity of solutions; see also the monographs \cite{DiBe,Urbano}.
Further, note that the lower bound on $p$ in \cite{Kinnunen-Lewis:1} appears naturally in different areas of parabolic regularity theory \cite{DiBe}.
In the meantime, \cite{Kinnunen-Lewis:1} has been generalized in several directions, including higher integrability results up to the parabolic boundary \cite{Boegelein-Parviainen,Parviainen,Parviainen-singular}, and results for higher order parabolic systems with $p$-growth \cite{Boegelein:1},
systems with $p(x,t)$-growth \cite{Boegelein-Duzaar}, and most recently parabolic double phase systems \cite{Kim-deg,Kim-sing}.

Despite this progress, higher integrability for the porous medium equation remained open for almost 20 years, since its nonlinearity concerns $u$ itself instead of its spatial gradient and is therefore significantly harder to deal with.
Then, Gianazza and Schwarzacher \cite{Gianazza-Schwarzacher} succeeded to prove the desired result for non-negative solutions to the degenerate porous medium equation by using intrinsic cylinders that depend on $u$ rather than $Du$.
The method in \cite{Gianazza-Schwarzacher} relies on the expansion of positivity.
Since this tool is only available for non-negative solutions, the approach does not carry over to sign-changing solutions or systems of porous medium type.
The case of systems was treated later by Bögelein, Duzaar, Korte and Scheven \cite{Boegelein-Duzaar-Korte-Scheven} for the transformed version of \eqref{eq:dne}
$$
	\partial_t u - \Div \A (x,t,u,D(|u|^{m-1}u)) = \Div F,
$$
where $m = \frac{1}{q} >0$, by using a different intrinsic geometry that also depends on $u$ itself.
Further, their proof of a reverse Hölder inequality is based on an energy estimate and the so-called gluing lemma, but avoids expansion of positivity.
Global higher integrability for degenerate porous medium type systems can be found in \cite{Moring-pme-global}.
For a local result concerning non-negative solutions in the supercritical singular range $\frac{(n-2)_+}{n+2}<m<1$, we refer to the paper \cite{Gianazza-Schwarzacher-singular} by Gianazza and Schwarzacher, and for sign-changing or vector-valued solutions to the article \cite{Boegelein-Duzaar-Scheven:2018} by Bögelein, Duzaar and Scheven.
Analogous to the observation for the singular parabolic $p$-Laplacian above, note that the lower bound $\frac{(n-2)_+}{n+2}$ is natural in the regularity theory for the fast diffusion equation, see \cite[Section 6.21]{DBGV-book}.

\begin{figure}
  \begin{tikzpicture}(160,155)(0,0)
	\draw[->,thick] (0,0) -- (7,0);
	\draw[->,thick] (0,0) -- (0,5.5);
	\draw (-0.05,1.5) -- (0.05,1.5);
	\draw (1.5,-0.05) -- (1.5,0.05);
	\draw (3,-0.05) -- (3,0.05);
	\draw (2,-0.05) -- (2,0.05);
	\draw (-0.2,1.5) node[anchor=mid] {$1$};
	\draw (1.5,-0.3) node[anchor=mid] {$1$};
	\draw (-0.4,4) node[anchor=mid] {$\frac{n+2}{n-2}$};
	\draw (3,-0.3) node[anchor=mid] {$2$};
	\draw (2,-0.3) node[anchor=mid] {$\frac{2n}{n+2}$};
	\draw (7,-0.3) node[anchor=mid] {$p$};
        \draw (5.9,5.2) node[anchor=mid] {\scriptsize $p=q+1$};
        \draw (-0.2,5.5) node[anchor=mid] {$q$};
    	\draw (2,2.8) node[anchor=mid] {\scriptsize $p=\frac{n(q+1)}{n+q+1}$};
	\draw (1.5,0) -- (7,5.5);
	\draw (0,1.5) -- (7,1.5);
	\draw (3,0) -- (3,5.5);
	\draw (2,0) -- (2,1.5);
        \draw [thick,dashdotted,black!30!red](2,1.5).. controls (2.5,2) and (3,3)
        .. (3,4);
        \draw [densely dotted] (0,4)--(3,4);
        \draw [thick](3,4)--(5.5,4);
        \draw [thick,dashed,black!30!green](5.5,4)--(7,4.6);
        \fill [opacity=0.5,red] (2,1.5).. controls (2.5,2) and (3,3)
        .. (3,4) -- (3,1.5);
        \fill [opacity=0.5,blue] (3,1.5)--(3,4)--(5.5,4);
        \fill [opacity=0.5,green] (3,1.5)--(5.5,4)--(7,4.6)--(7,1.5);
	\fill [opacity=0.5,lightgray] (3,0) -- (7,0) -- (7,1.5) -- (3,1.5);
	\fill [opacity=0.5,lightgray] (2,0) -- (3,0) -- (3,1.5) -- (2,0.5);
	\fill [opacity=0.5,lightgray] (3,1.5) -- (2,0.5) -- (2,1.5);
        \draw (6.35,4) node[anchor=mid] {\scriptsize
          $q=\frac{2p+n}{n}$};
\end{tikzpicture}
\caption{The red, blue and green areas are the ranges of $p$ and $q$ covered by Theorem \ref{thm:main}.}\label{figure}
\end{figure}
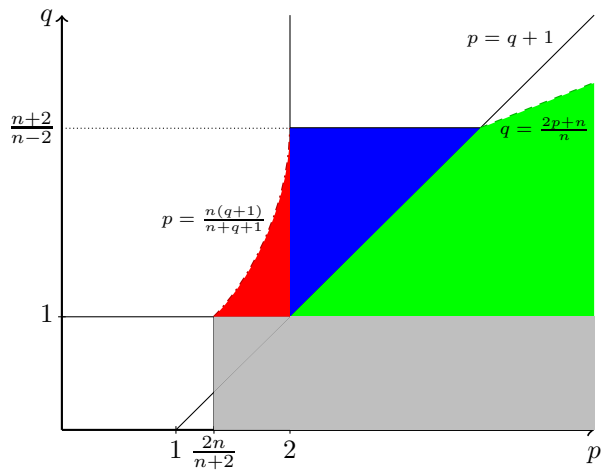

As a next step, Bögelein, Duzaar, Kinnunen and Scheven \cite{BDKS:2018} proved local higher integrability for the system \eqref{eq:dne} in the homogeneous case $p=q+1$.
To this end, they developed a new, elaborate intrinsic geometry that depends on both $u$ and $Du$, thus reflecting the doubly nonlinear behavior of the system.
The range $\max\big\{1,\frac{2n}{n+2}\big\} <p< \frac{2n}{(n-2)_+}$ of their main result seems unexpected first; however, the lower bound is the natural one for the parabolic $p$-Laplacian, while the upper bound is the same as for the singular porous medium system (note that it can be expressed as $q = p-1 < \frac{n+2}{(n-2)_+}$).
For $N=1$, non-negative solutions and $F \equiv 0$, Saari and Schwarzacher \cite{Saari-Schwarzacher} were able to remove the upper bound for all dimensions $n \in \N$.
Finally, the range $0<q<1$ and $\frac{2n}{n+2} < p$ of \eqref{eq:dne},
i.e.~the degenerate case with respect to $u$, has been dealt with by
Bögelein, Duzaar and Scheven in
\cite{Boegelein-Duzaar-Scheven:2022}.
The range covered by
\cite{Boegelein-Duzaar-Scheven:2022} corresponds to the gray area in
Figure \ref{figure}.

The goal of the present paper is to treat the singular range $q>1$ and thus close the gap in the higher integrability theory for \eqref{eq:dne}.
The overall strategy is similar to the one in \cite{Boegelein-Duzaar-Scheven:2022}.
However, there is a crucial difference in the chosen intrinsic geometry.
While scaling in the time variable is appropriate in the degenerate
case, the technique seems to require a different scaling in the
singular case. Thus, we work with a scaling both in the spatial and time variables.
Namely, throughout the article we consider cylinders of the form
$$
	Q_\varrho^{(\lambda,\theta)}(x_o,t_o)
	:=
	B_{\theta^\frac{1-q}{1+q} \varrho}(x_o)
	\times (t_o - \lambda^{2-p} \varrho^{1+q}, t_o + \lambda^{2-p} \varrho^{1+q})
$$
with positive factors $\lambda,\theta$ and $(x_o,t_o) \in \Omega_T$.
We collect technical lemmas, energy estimates and the gluing lemma for such cylinders in Section~\ref{sec:preliminaries}.
In particular, the latter two have already been proven in \cite{Boegelein-Duzaar-Scheven:2022} for all $p>1$ and $q>0$.
Now, the idea is to select $\lambda$ and $\theta$ such that
\begin{equation}
	\lambda^p \approx \biint_{Q_\varrho^{(\lambda,\theta)}} |Du|^p + |F|^p \,\dx\dt
	\quad \text{and} \quad
	\theta^{p^\sharp} \approx
        \biint_{Q_\varrho^{(\lambda,\theta)}}
        \frac{|u|^{p^\sharp}}{\big(\theta^{\frac{1-q}{1+q}}\varrho\big)^{p^\sharp}} \,\dx\d t
       	\label{eq:idea-scaling}
\end{equation}
in order to obtain intrinsic cylinders.
However, due to some complications related to their construction, we
also need to take so-called $\theta$-subintrinsic cylinders into
account, where only the inequality "$\gtrsim$" is satisfied in
\eqref{eq:idea-scaling}$_2$.
More precisely, we can construct cylinders in such a way that they are
either $\theta$-intrinsic in the sense of~\eqref{eq:idea-scaling}$_2$
or that they are $\theta$-subintrinsic and satisfy
$\theta\lesssim\lambda$, see \eqref{eq:theta-degenerate-1}.
We call the latter case $\theta$-singular because it means that $u$ is in a certain
sense small compared to its oscillation, and the differential equation
becomes singular if $|u|$ becomes small.
In both cases, sophisticated arguments are necessary to prove parabolic Sobolev-Poincaré type inequalities for all relevant cylinders.
This is done in the regime $\frac{n(q+1)}{n+q+1} < p \leq q+1$ in Section~\ref{sec:sobolev-fast-diffusion} and in the range $2 < q+1 < p$ in Section~\ref{sec:sobolev-slow-diffusion}.
Reverse Hölder inequalities in the same types of cylinders are shown for the whole range $q>1$ and $\frac{n(q+1)}{n+q+1} < p$ in Section~\ref{sec:reverse-Holder}.
The lower bound on $p$ appearing in the proof of these vital tools and thus restricting the red area of admissible parameters in Figure~\ref{figure} is natural in the regularity theory of the doubly nonlinear equation \eqref{eq:prototype}.
Finally, the proof of Theorem~\ref{thm:main} can be found in Section~\ref{sec:proof-main-thm}.
To this end, we start with a given non-intrinsic cylinder $Q_{2R} \Subset \Omega_T$ and first focus on the second relation in \eqref{eq:idea-scaling} in Section~\ref{subsec:cylinder-construction}.
This is the step where, in the case $n\ge3$, the conditions $q<\frac{n+2}{n-2}$ for $p<q+1$ and $q<\frac{2p}{n}+1$ for $p>q+1$ restricting the blue and green parameter areas in Figure~\ref{figure} come into play.
These conditions are consistent with the bounds $q <\frac{n+2}{n-2}$
for the singular porous medium system in
\cite{Boegelein-Duzaar-Scheven:2018} and $q+1=p<\frac{2n}{n-2}$ for
the homogeneous doubly nonlinear system in \cite{BDKS:2018}.
Even in
the latter special case, it remains an interesting open problem to
remove this condition in the case of systems.

Ideally, we would like to choose $\theta$ in dependence on given parameters $\lambda$ and $\varrho$ such that $\varrho \mapsto \theta$ (with fixed $\lambda$) is non-increasing and that $Q_\varrho^{(\lambda,\theta)}\subset Q_{2R}$ satisfies \eqref{eq:idea-scaling}$_2$.
The reason that it is only possible to obtain $\theta$-subintrinsic cylinders is the so-called sunrise construction that is used to ensure the monotonicity of $\varrho \mapsto \theta$.
Next, we prove a Vitali-type covering property for the relevant cylinders in Section~\ref{sec:Vitali-covering}.
In Section~\ref{sec:stopping-time}, for given $\lambda$ we use a stopping time argument to fix the radius of our (sub)-intrinsic cylinders (and thus the parameter $\theta$ according to the first step) such that also the first relation in \eqref{eq:idea-scaling} is satisfied.
Applying the results of Section~\ref{sec:reverse-Holder}, we show that a suitable reverse Hölder inequality holds in Section~\ref{sec:reverse-Holder-final-section}.
Finally, we sketch standard arguments that finish the proof in Section~\ref{sec:final-argument}.

\medskip

\noindent
{\bf Acknowledgments.} K.~Moring has been supported by the Magnus Ehrnrooth Foundation.
L.~Schätzler was partly supported by the FWF-Project P31956-N32 ``Doubly
nonlinear evolution equation''.
Further, she would like to express her gratitude to the Faculty of Mathematics of the University of Duisburg-Essen for the hospitality during her visit.

 \section{Preliminaries}
\label{sec:preliminaries}

We write $z_o = (x_o,t_o) \in \R^n \times \R$ and use space-time cylinders of the form
$$
Q_\rho^{(\lambda, \theta)}(z_o) = B_\rho^{(\theta)}(x_o) \times \Lambda_\rho^{(\lambda)}(t_o),
$$
where
$$
B_\rho^{(\theta)}(x_o) = \left\{ x\in \R^n : |x-x_o| < \theta^\frac{1-q}{1+q} \rho \right\},
$$
and
$$
\Lambda^{(\lambda)}_\rho (t_o) = \left(t_o - \lambda^{2-p} \rho^{1+q}, t_o + \lambda^{2-p} \rho^{1+q} \right),
$$
with parameters $\theta, \lambda > 0$. If $\lambda = \theta = 1$, we use the simpler notation
$$
	Q_\varrho(z_o) := Q_\varrho^{(1,1)}(z_o).
$$
For the mean value of a function $u\in L^1(Q)$ over a cylinder
$Q=B\times\Lambda\subset\R^{n}\times\R$ of finite positive measure, we write 
\begin{equation*}
  (u)_Q
  :=
  \biint_{Q}u\,\d x\d t
\end{equation*}
and similarly,
\begin{equation*}
  (u)_B(t)
  :=
  \bint_{B}u(\cdot,t)\,\d x
\end{equation*}
for the slice-wise means, provided $u(\cdot,t)\in L^1(B)$. In the particular cases
$Q=Q_\rho^{(\lambda, \theta)}(z_o)$ and $B=B_\rho^{(\theta)}(x_o)$, we
also write 
\begin{equation*}
  (u)_{z_o;\rho}^{(\lambda, \theta)}
  :=
  (u)_{\rho}^{(\lambda, \theta)}
  :=
  (u)_Q
  \qquad\mbox{and}\qquad 
  (u)_{x_o;\rho}^{(\theta)}(t)
  :=
  (u)_{\rho}^{(\theta)}(t)
  :=
  (u)_B(t).
\end{equation*}
For the power of a vector $u\in\R^N$ to an exponent $\alpha>0$, we write
\begin{equation*}
  \u^\alpha:=|u|^{\alpha-1}u,
\end{equation*}
where we interpret the right-hand side as zero if $u=0$.

Next we state a useful iteration lemma that can be obtained by a change of variables in \cite[Lemma 6.1]{Giusti:book}.
\begin{lemma}
\label{lem:iteration} 
Let $0<\vartheta<1$, $A,C\geq 0$ and $\alpha,\beta>0$. Then, there
exists a constant $c=c(\alpha,\beta,\vartheta)$ such that there holds: For any $0<r<\rho$ and any nonnegative bounded function $\phi\colon [r,\rho] \to \R_{\geq 0}$ satisfying
$$
\phi(t) \leq \vartheta \phi(s) +A(s^\alpha-t^\alpha)^{-\beta}+C \quad \text{ for all } r\leq t <s \leq \rho,
$$
we have
$$
\phi(r) \leq c \left[A(\rho^\alpha-r^\alpha)^{-\beta}+C \right].
$$
\end{lemma}

Using the arguments of \cite[Lemma 8.3]{Giusti:book}, the following lemma can be deduced.
\begin{lemma}\label{lem:Acerbi-Fusco}
For every $\alpha>0$, there exists a constant $c=c(\alpha)$ such that,
for all $a,b\in\R^N$, $N\in\N$, we have
\begin{align*}
	\tfrac1c\big|\power{b}{\alpha} - \power{a}{\alpha}\big|
	\le
	\big(|a| + |b|\big)^{\alpha-1}|b-a|
	\le
	c \big|\power{b}{\alpha} - \power{a}{\alpha}\big|.
\end{align*}
\end{lemma}

In the case $\alpha\ge1$, the preceding lemma immediately implies the
following elementary estimate.
\begin{lemma}\label{lem:a-b}
For every $\alpha\ge 1$, there exists a constant $c=c(\alpha)$ such that,
for all $a,b\in\R^N$, $N\in\N$, we have
\begin{align*}
	|b-a|^\alpha
    \le
    c\big|\power{b}{\alpha} - \power{a}{\alpha}\big|.
\end{align*}
\end{lemma}

For the proof of the following statement on the quasi-minimality of the mean value, we refer to \cite[Lemma 3.5]{BDKS:2018}.
\begin{lemma}
\label{lem:uavetoa}
Let $p \geq 1$ and $\alpha \geq \frac1p$. There exists a constant $c = c(\alpha,p)$ such that whenever $A \subset B \subset \R^k$, $k \in \N$ holds for bounded sets $A$ and $B$ of positive measure, then for every $u \in L^{\alpha p}(B,\R^N)$ and $a \in \R^N$ there holds 
\begin{align*}
\bint_B \babs{\u^{\alpha} - \lbbracket \u \rbbracket_{\bf{A}}^{\alpha}}^p\, \dx \leq \frac{c|B|}{|A|} \bint_B \babs{\u^{\alpha} - \a^{\alpha}}^p\, \dx.
\end{align*}
\end{lemma}

Next, we recall the Gagliardo-Nirenberg inequality.
\begin{lemma} \label{lem:GN}
Let $1 \leq p,q, r < \infty$ and $\vartheta \in (0,1)$ such that $-\frac{n}{p} \leq \vartheta (1- \frac{n}{q}) - (1-\vartheta) \frac{n}{r}$. Then, there exists a constant $c = c(n,p)$ such that for any ball $B_\rho(x_o) \subset \R^n$ with $\rho > 0$ and any function $u \in W^{1,q}(B_\rho(x_o))$ we have
$$
\bint_{B_\rho(x_o)} \frac{|u|^p}{\rho^p} \, \d x \leq c \left[ \bint_{B_\rho(x_o)} \left( \frac{|u|^q}{\rho^q} + |Du|^q \right) \, \d x \right]^\frac{\vartheta p}{q} \left[ \bint_{B_\rho(x_o)} \frac{|u|^r}{\rho^r} \, \d x \right]^\frac{(1-\vartheta) p}{r}.
$$
\end{lemma}

Finally, the proof of the following two lemmas can be found in
\cite{Boegelein-Duzaar-Scheven:2022}. We note that in
\cite{Boegelein-Duzaar-Scheven:2022}, a slightly different definition
of intrinsic cylinders has been used. In order to obtain the following
statements, we replace the radii $\rho,r$ in
\cite{Boegelein-Duzaar-Scheven:2022} by
$\theta^{\frac{1-q}{1+q}}\rho,\theta^{\frac{1-q}{1+q}} r$. 
We start with an energy estimate for solutions of~\eqref{eq:dne}.

\begin{lemma}[\!\!{\cite[Lemma 3.1]{Boegelein-Duzaar-Scheven:2022}}]
\label{lem:caccioppoli}
Let $p>1$, $q > 0$ and $u$ be a weak solution to \eqref{eq:dne} where the vector field $\A$ satisfies \eqref{assumption:A}. Then, there exists a constant $c=c(p,q,C_o,C_1)$ such that on every cylinder $Q_\rho^{(\lambda,\theta)}(z_o)\Subset \Omega_T$ with $\rho>0$ and $\lambda, \theta>0$ and for any $r\in [\rho/2,\rho)$ and all $a \in \R^N$ the following energy estimate
\begin{align}\label{energy-estimate}\nonumber
\sup_{t\in \Lambda_r^{(\lambda)}(t_o)}&\, \bint_{B_r^{(\theta)}(x_o)} \frac{\babs{ \u^\frac{q+1}{2}(t)-\a^\frac{q+1}{2} }^2}{\lambda^{2-p} r^{q+1}} \d x + \biint_{Q_r^{(\lambda,\theta)}(z_o)} |Du|^p\, \d x\d t \\
& \leq c \biint_{Q_\rho^{(\lambda,\theta)}(z_o)} \left[ \theta^\frac{p(q-1)}{q+1} \frac{|u-a |^p}{(\rho-r)^p}+ \frac{\babs{ \u^\frac{q+1}{2}-\a^\frac{q+1}{2} }^2}{ \lambda^{2-p} (\rho^{q+1}-r^{q+1} )} + |F|^p  \right] \d x\d t 
\end{align}
holds true.
\end{lemma}

Then we state the gluing lemma.

\begin{lemma}[\!\!{\cite[Lemma 3.2]{Boegelein-Duzaar-Scheven:2022}}] \label{lem:gluing}
Let $p>1$, $q > 0$ and $u$ be a weak solution to \eqref{eq:dne} where the vector field $\A$ satisfies \eqref{assumption:A}. Then, there exists a constant $c=c(C_1)$ such that on every cylinder $Q_{\rho}^{(\lambda,\theta)}(z_o)\Subset \Omega_T$ with $\rho>0$ and $\lambda, \theta>0$ there exists $\hat \rho \in [\frac{\rho}{2},\rho]$ such that for all $t_1,t_2 \in \Lambda_\rho^{(\lambda)}(t_o)$ there holds
$$
\babs{ (\u^q)_{\hat \rho}^{(\theta)}(t_2) - (\u^q)_{\hat \rho}^{(\theta)}(t_1) } \leq c \lambda^{2-p} \theta^\frac{q-1}{q+1} \rho^q \biint_{Q_\rho^{(\lambda,\theta)}} \left( |Du|^{p-1} + |F|^{p-1} \right) \, \d x \d t.
$$
\end{lemma}

\section{Parabolic Sobolev-Poincar\'e type inequalities in case \texorpdfstring{$q+1 \geq p$}{q+1>=p}}
\label{sec:sobolev-fast-diffusion}

The goal of this section is to prove Sobolev-Poincar\'e inequalities that bound the right-hand side of the energy estimate~\eqref{energy-estimate} from above.
It turns out that different strategies are required for the cases $q+1\geq p$ and $q+1<p$. 
Therefore, we only consider the first case here and postpone the second one to the next section.

We use $\lambda$-intrinsic
\begin{equation} \label{eq:lambda-intrinsic}
\frac{1}{C_\lambda} \biint_{Q_{2\rho}^{(\lambda,\theta)}}  |Du|^p + |F|^p \, \d x \d t \leq \lambda^p\leq C_\lambda \biint_{Q_{\rho}^{(\lambda,\theta)}}  |Du|^p + |F|^p \, \d x \d t,
\end{equation}
$\theta$-intrinsic
\begin{equation} \label{eq:theta-intrinsic}
\frac{1}{C_\theta} \biint_{Q_{2\rho}^{(\lambda,\theta)}}  \frac{|u|^{p^\sharp}}{(2\rho)^{p^\sharp}} \, \d x \d t \leq \theta^{\frac{2p^\sharp}{q+1}} \leq C_\theta \biint_{Q_{\rho}^{(\lambda,\theta)}}  \frac{|u|^{p^\sharp}}{\rho^{p^\sharp}} \, \d x \d t
\end{equation}
scalings, in which $p^\sharp = \max \{p,q+1\}$.
However, for the cylinders constructed in
Section~\ref{subsec:cylinder-construction}, we are not able to prove
the $\theta$-intrinsic scaling in every case. In general, we can only prove
the first of the two inequalities in~\eqref{eq:theta-intrinsic}, which we refer to as
$\theta$-subintrinsic scaling. In
Section~\ref{sec:reverse-Holder-final-section} we will show that 
the cylinders used in the proof either satisfy the $\theta$-intrinsic
scaling~\eqref{eq:theta-intrinsic} or
a scaling of the form
\begin{equation} \label{eq:theta-degenerate-1}
\frac{1}{C_\theta} \biint_{Q_{2\rho}^{(\lambda,\theta)}}  \frac{|u|^{p^\sharp}}{(2\rho)^{p^\sharp}} \, \d x \d t \leq \theta^{\frac{2p^\sharp}{q+1}} \leq C_\theta \left( \biint_{Q_{\rho}^{(\lambda,\theta)}} |Du|^p + |F|^p \, \d x \d t \right)^\frac{2 p^\sharp}{p(q+1)}.
\end{equation}
We call this scaling $\theta$-singular because it means that the
solution is in a certain sense small compared to its oscillation, in which case the
differential equation~\eqref{eq:dne} becomes singular.

For now, we suppose that $q+1 \geq p$. Then~\eqref{eq:theta-intrinsic} reads as
\begin{equation} \label{eq:theta-intrinsic-q+1}
\frac{1}{C_\theta} \biint_{Q_{2\rho}^{(\lambda,\theta)}}  \frac{|u|^{q+1}}{(2\rho)^{q+1}} \, \d x \d t \leq \theta^{2} \leq C_\theta \biint_{Q_{\rho}^{(\lambda,\theta)}}  \frac{|u|^{q+1}}{\rho^{q+1}} \, \d x \d t
\end{equation}
and~\eqref{eq:theta-degenerate-1} as
\begin{equation} \label{eq:theta-degenerate}
\frac{1}{C_\theta} \biint_{Q_{2\rho}^{(\lambda,\theta)}}  \frac{|u|^{q+1}}{(2\rho)^{q+1}} \, \d x \d t \leq \theta^{2} \leq C_\theta \left( \biint_{Q_{\rho}^{(\lambda,\theta)}} |Du|^p + |F|^p \, \d x \d t \right)^\frac{2}{p}.
\end{equation}

We start with a Sobolev-Poincar\'e type estimate for the second term
appearing on the right-hand side of the energy estimate from
Lemma~\ref{lem:caccioppoli}.

\begin{lemma} \label{lem:sobo-poincare-q+1}
  Suppose that $q > 1$, $\frac{n(q+1)}{n+q+1}< p \leq q+1$,
  and that $u$ is a weak solution to~\eqref{eq:dne}, under
  assumption~\eqref{assumption:A}. Moreover, we consider a cylinder
$Q_{2\rho}^{(\lambda,\theta)}(z_o) \Subset \Omega_T$ and assume 
that~\eqref{eq:lambda-intrinsic} is satisfied together with either~\eqref{eq:theta-intrinsic-q+1} or~\eqref{eq:theta-degenerate}. Then the following Sobolev-Poincar\'e inequality holds:
\begin{align*}
  \lambda^{p-2} &\biint_{Q_\rho^{(\lambda,\theta)}(z_o)}
          \frac{\babs{ \u^\frac{q+1}{2}-\a^\frac{q+1}{2} }^2}{ \rho^{q+1} } \, \d x \d t \\
&\leq \eps \left( \sup_{t \in \Lambda_\rho^{(\lambda)}(t_o)} \lambda^{p-2} \bint_{B_\rho^{(\theta)}(x_o)} \frac{\babs{ \u^\frac{q+1}{2}(t)-\a^\frac{q+1}{2} }^2}{ \rho^{q+1} } \, \d x + \biint_{Q_\rho^{(\lambda,\theta)}(z_o)} |Du|^p \, \d x \d t \right) \\
&\phantom{+} + c \eps^{-\beta} \left[ \left( \biint_{Q_\rho^{(\lambda,\theta)}(z_o)} |Du|^{\nu p} \, \d x \d t  \right)^\frac{1}{\nu} + \biint_{Q_\rho^{(\lambda,\theta)}(z_o)} |F|^p \, \d x \d t  \right],
\end{align*}
where $\max \left\{\frac{n(q+1)}{p(n+q+1)} , \frac{p-1}{p}
\right\} \leq \nu \leq 1$ and $a = (u)_{z_o;\rho}^{(\theta,\lambda)}$. The preceding
estimate holds for an arbitrary $\eps\in(0,1)$ with a constant
$c = c(n,p,q,C_1,C_\theta,C_\lambda)> 0$ and $\beta = \beta(n,p,q) > 0$.
\end{lemma}

\begin{proof}
Since the cylinder is fixed throughout the proof, we
use the more compact notations $Q:=Q_\rho^{(\theta,\lambda)}(z_o)$,
$B:=B_\rho^{(\theta)}(x_o)$ and
$\Lambda:=\Lambda_\rho^{(\lambda)}(t_o)$. Furthermore, with the radius
$\hat\rho\in[\frac\rho2,\rho]$ provided by Lemma~\ref{lem:gluing}, we
write $\widehat B:=B_{\hat\rho}^{(\theta)}(x_o)$ and $\widehat
Q:=\widehat B\times\Lambda$. 
Using first Lemma~\ref{lem:uavetoa} with $\alpha=\frac{q+1}{2}$ and $p=2$
and then the triangle inequality, we estimate 
\begin{align}\label{def-I-II}
&\lambda^{p-2} \biint_{Q}
  \frac{|\u^\frac{q+1}{2} - \a^\frac{q+1}{2}|^2}{\rho^{q+1}} \, \d x \d t \\\nonumber
&\qquad\leq 
c\lambda^{p-2} \biint_{Q} \frac{\big|\u^\frac{q+1}{2} - {[(\u^q)_{\widehat B}(t)]}^\frac{q+1}{2q}\big|^2}{\rho^{q+1}} \, \d x \d t \\\nonumber
  &\qquad\phantom{+} +
    c\lambda^{p-2} \bint_{\Lambda}
    \frac{\big|[(\u^q)_{\widehat B}(t)]^\frac{q+1}{2q} -
    [(\u^q)_{\widehat Q}]^\frac{q+1}{2q}\big|^2}{\rho^{q+1}} \, \d t\\\nonumber
&\qquad=: \mathrm{I} + \mathrm{II}.
\end{align}

We use Lemma~\ref{lem:uavetoa} with $\alpha=\frac{q+1}{2q}$ and $p=2$ to estimate
\begin{align}\label{bound-I}
\mathrm I &\leq \frac{c\lambda^{p-2}}{\rho^{q+1}} \sup_{t\in\Lambda} \left[ \bint_B
            \big|\u^\frac{q+1}{2} - {[(\u^q )_{\widehat B}(t)]}^\frac{q+1}{2q}\big|^2 \, \d x \right]^{\frac{2}{n+2}} \\\nonumber
&\phantom{+} \cdot \bint_{\Lambda} \left[ \bint_B
                                                                                                                                    \big|\u^\frac{q+1}{2} - {[(\u^q )_{\widehat B}(t)]}^\frac{q+1}{2q}\big|^2  \, \d x \right]^{\frac{n}{n+2}} \, \d t \\\nonumber
&\leq \eps \, \sup_{t \in \Lambda} \lambda^{p-2} \bint_B \frac{\big|\u^\frac{q+1}{2} - \a^\frac{q+1}{2} \big|^2}{\rho^{q+1}} \, \d x \\\nonumber
&\phantom{+} + \frac{c \lambda^{p-2}}{\eps^{\frac{2}{n}} \rho^{q+1}}
                                                                                                                                                   \left( \bint_{\Lambda} \left[ \bint_B |\u^\frac{q+1}{2} - {[(u)_{B}(t)]}^\frac{q+1}{2}|^2 \, \d x	 \right]^{\frac{n}{n+2}} \, \d t \right)^\frac{n+2}{n} \\\nonumber
&=: \eps \, \sup_{t\in \Lambda} \lambda^{p-2} \bint_B \frac{\big|\u^\frac{q+1}{2} - \a^\frac{q+1}{2} \big|^2}{\rho^{q+1}} \, \d x + \frac{c \lambda^{p-2}}{\eps^{\frac{2}{n}} \rho^{q+1}} \mathrm{III}.
\end{align}
In the last inequality, we also used Young's inequality with exponents
$\frac{n+2}{2}$ and $\frac{n+2}{n}$. 
Observe that Lemma~\ref{lem:Acerbi-Fusco} and H\"older's inequality imply
\begin{align*}
\bint_B &|\u^\frac{q+1}{2} - {[(u)_{B}(t)]}^\frac{q+1}{2}|^2 \, \d x \\
&\leq c \bint_B \left( |u| + |(u)_{B}(t)| \right)^{q-1} |u - (u )_{B}(t)|^2 \, \d x \\
&\leq c \left( \bint_B |u|^{q+1} \, \d x \right)^\frac{q-1}{q+1} \left( \bint_B |u - (u)_{B}(t)|^{q+1} \, \d x \right)^\frac{2}{q+1}.
\end{align*}
By applying H\"older inequality in the time integral with exponents
$\frac{n+2}{n}\cdot\frac{q+1}{q-1}$ and $\frac{n+2}{2}\cdot\frac{q+1}{n+q+1}$ we obtain
\begin{align*}
\mathrm{III} \leq c \left( \biint_Q |u|^{q+1} \, \d x \d t
  \right)^{\frac{q-1}{q+1}}
  \left( \bint_{\Lambda} \left[ \bint_B |u - (u )_{B}(t)|^{q+1} \, \d x \right]^\frac{n}{n+q+1} \, \d t \right)^{\frac{2}{n}\frac{n+q+1}{q+1}}.
\end{align*}
By $\theta$-subintrinsic scaling 
$$
\left( \biint_Q |u|^{q+1} \, \d x \d t \right)^\frac{q-1}{q+1} \leq c \rho^{q-1} \theta^{\frac{2(q-1)}{q+1}},
$$
and by Sobolev inequality we have
\begin{align*}
  \left[ \bint_B |u - (u )_{B}(t)|^{q+1} \, \d x \right]^\frac{n}{n+q+1} 
  \leq
  c \left( \theta^\frac{1-q}{1+q} \rho \right)^\frac{n(q+1)}{n+q+1}\bint_B |Du|^\frac{n(q+1)}{n+q+1} \, \d x .
\end{align*}
We combine the estimates and obtain
\begin{align}\label{bound-III}
  \mathrm{III}
  \leq
  c \rho^{q+1} \left( \biint_{Q}
  |Du|^\frac{n(q+1)}{n+q+1} \, \d x \d t
  \right)^\frac{2(n+q+1)}{n(q+1)}
  \le
  c \rho^{q+1} \left( \biint_{Q}
  |Du|^{\nu p}\, \d x \d t \right)^\frac{2}{\nu p}.
\end{align}
The last estimate follows from H\"older's inequality, since $\nu p\ge\frac{n(q+1)}{n+q+1}$.
In the case $p<2$, we use the $\lambda$-subintrinsic
scaling \eqref{eq:lambda-intrinsic}$_1$ and H\"older's inequality,
which yields the bound
\begin{align*}
  \lambda\ge c\left( \biint_{Q} |Du|^{\nu p} \, \d x \d t
  \right)^\frac{1}{\nu p},
\end{align*}
while in the case $p\ge 2$, we use Young's
inequality. In both cases, we observe that~\eqref{bound-III} implies 
\begin{align*}
  \frac{c\lambda^{p-2}}{\eps^{\frac{2}{n}} \rho^{q+1}} \mathrm{III}
  \le
  \eps\lambda^p+c\eps^{-\beta}\left( \biint_{Q} |Du|^{\nu p} \, \d x \d t \right)^\frac{1}{\nu},
\end{align*}
where the term $\eps\lambda^p$ can be omitted in the case $p<2$.
Here and in the remainder of the proof, we write $\beta$ for a
positive universal constant that depends at most on $n,p$ and $q$.
Bounding the right-hand side by the $\lambda$-superintrinsic scaling
\eqref{eq:lambda-intrinsic}$_2$ and using the resulting estimate to
bound the right-hand side of~\eqref{bound-I} from above, we deduce
\begin{align}\label{bound-I-final}
\mathrm I &\leq c\eps \left(  \sup_{t\in \Lambda} \lambda^{p-2} \bint_B \frac{\big|\u^\frac{q+1}{2} - \a^\frac{q+1}{2}\big|^2}{\rho^{q+1}} \, \d x + \biint_Q |Du|^p \, \d x \d t \right) \\\nonumber
&\phantom{+} + c\eps^{-\beta} \left( \biint_Q |Du|^{\nu p} \, \d x \d t \right)^\frac{1}{\nu} + c\, \biint_Q |F|^p \, \d x \d t.
\end{align}

Then let us turn our attention to the term $\mathrm{II}$. We apply in
turn Lemma~\ref{lem:a-b} with $\alpha=\frac{2q}{q+1}\ge1$ and then 
Lemma~\ref{lem:gluing} to get
\begin{align}\label{bound-II-theta}
\mathrm{II} &\le c\lambda^{p-2} \bint_{\Lambda}
    \frac{\big| {(\u^q)_{\widehat B}}(t) -
              (\u^q)_{\widehat Q}\big|^{\frac{q+1}{q}}}{\rho^{q+1}} \, \d t\\\nonumber
  &\leq c \lambda^{p-2} \bint_{\Lambda} \bint_{\Lambda} \frac{\big|
    {(\u^q)_{\hat \rho}^{(\theta)}}(t) - (\u^q)_{\hat \rho}^{(\theta)}(\tau)\big|^\frac{q+1}{q}}{\rho^{q+1}} \, \d t \d \tau \\\nonumber
&\leq c \lambda^\frac{2-p}{q} \theta^\frac{q-1}{q} \left( \biint_Q |Du|^{p-1} + |F|^{p-1} \, \d x \d t \right)^\frac{q+1}{q}.
\end{align}
In the case~\eqref{eq:theta-intrinsic-q+1} we
estimate
\begin{align*}
  \theta^2
  &\leq c \biint_Q \frac{|\u^\frac{q+1}{2} - [(\u^q)_{\widehat
  Q}]^\frac{q+1}{2q}|^2}{\rho^{q+1}} \, \d x \d t + c
  \frac{|(\u^q)_{\widehat Q}|^\frac{q+1}{q}}{\rho^{q+1}}\\
  &\le
    c \biint_Q \frac{|\u^\frac{q+1}{2} - \a^\frac{q+1}{2}|^2}{\rho^{q+1}} \, \d x \d t + c
  \frac{|(\u^q)_{\widehat Q}|^\frac{q+1}{q}}{\rho^{q+1}},
\end{align*}
where we used Lemma~\ref{lem:uavetoa} with $\alpha=\frac{q+1}{2q}$ and
$p=2$ in the last step.
We use this to estimate
\begin{align*}
\mathrm{II}
  =
  \frac{\theta^{\frac{2(q-1)}{q+1}}}{\theta^{\frac{2(q-1)}{q+1}}}  \mathrm{II}
  \leq
  \mathrm{II}_1 + \mathrm{II}_2,
\end{align*}
where we denoted
\begin{align*}
\mathrm{II}_1 := \frac{c}{\theta^\frac{2(q-1)}{q+1}} \left[ \biint_Q
  \frac{\big|\u^\frac{q+1}{2} - \a^\frac{q+1}{2}\big|^2}{\rho^{q+1}} \, \d x \d t \right]^\frac{q-1}{q+1}  \cdot \mathrm{II}
\end{align*}
and
\begin{align*}
\mathrm{II}_2 := \frac{c|(\u^q)_{\widehat Q}|^\frac{q-1}{q}}{\theta^\frac{2(q-1)}{q+1} \rho^{q-1}} \cdot \mathrm{II}.
\end{align*}
For the estimate of $\mathrm{II}_1$, we use in turn~\eqref{bound-II-theta}, the
$\theta$-subintrinsic scaling and then Young's inequality with
exponents $\frac{2q}{q-1}$ and $\frac{2q}{q+1}$, with the result   
\begin{align*}
  \mathrm{II}_1
  &\leq c \lambda^\frac{2-p}{q} \theta^{-\frac{(q-1)^2}{q(q+1)}} \left[ \biint_Q
    \frac{\big|\u^\frac{q+1}{2} - \a^\frac{q+1}{2}\big|^2}{\rho^{q+1}} \, \d x \d t
    \right]^\frac{q-1}{q+1} \\\nonumber
  &\qquad\cdot\left[ \biint_Q |Du|^{p-1} + |F|^{p-1} \, \d x \d t \right]^\frac{q+1}{q} \\\nonumber
  &\leq c \left[ \lambda^{p-2}\biint_Q
    \frac{\big|\u^\frac{q+1}{2} - \a^\frac{q+1}{2}\big|^2}{\rho^{q+1}} \, \d x \d t
    \right]^\frac{q-1}{2q}\\\nonumber
  &\qquad\cdot
    \lambda^\frac{(2-p)(q+1)}{2q}\left[ \biint_Q |Du|^{p-1} + |F|^{p-1} \, \d x \d t \right]^\frac{q+1}{q} \\\nonumber
&\leq \frac{1}{2} \lambda^{p-2} \biint_Q \frac{\big|\u^\frac{q+1}{2} -\a^\frac{q+1}{2}\big|^2}{\rho^{q+1}} \, \d x \d t 
 + c \lambda^{2-p} \left[ \biint_Q |Du|^{p-1} + |F|^{p-1} \, \d x \d t \right]^2.
\end{align*}
Using the definition of $\mathrm{II}$ and Lemma~\ref{lem:a-b}, we also have 
\begin{align*}
  \mathrm{II}_2
  &\le
    \frac{c \lambda^{p-2}}{\theta^\frac{2(q-1)}{q+1}\rho^{2q}}
    \bint_{\Lambda} \big|(\u^q)_{\widehat B}(t) - (\u^q)_{\widehat Q}\big|^2 \, \d t\\\nonumber
  &\leq \frac{c \lambda^{p-2}}{\theta^\frac{2(q-1)}{q+1}\rho^{2q}} \bint_\Lambda\bint_\Lambda \big|(\u^q)_{\hat
                \rho}^{(\theta)}(t) - (\u^q)_{\hat
    \rho}^{(\theta)}(\tau)\big|^2  \, \d t \d \tau\\\nonumber
  &\leq c \lambda^{2-p} \left[ \biint_Q |Du|^{p-1} + |F|^{p-1} \, \d x \d t \right]^2.
\end{align*}
In the last step, we used Lemma~\ref{lem:gluing}.
We combine the two preceding estimates to
\begin{align}\label{bound-II-degenerate}
  \mathrm{II}
  &\le
    \frac1{2} \lambda^{p-2} \biint_Q \frac{\big|\u^\frac{q+1}{2} -\a^\frac{q+1}{2}\big|^2}{\rho^{q+1}} \, \d x \d t \\\nonumber
&\phantom{+} + c \lambda^{2-p} \left[ \biint_Q |Du|^{p-1} + |F|^{p-1} \, \d x \d t \right]^2.
\end{align}
In order to estimate the last term further, we distinguish
between the cases $p \geq 2$ and $p<2$. In the first case, we use the
$\lambda$-intrinsic scaling~\eqref{eq:lambda-intrinsic}, which implies
$$
  \lambda \geq c \left[ \biint_Q |Du|^{p-1} + |F|^{p-1} \, \d x \d t
  \right]^\frac{1}{p-1}.
$$
In the case $p<2$, we apply Young's inequality with exponents
$\frac{p}{2-p}$ and $\frac{p}{2(p-1)}$. In both cases, we deduce that~\eqref{bound-II-degenerate} implies
\begin{align}\label{bound-II-final}
\mathrm{II}
&\leq
\eps \lambda^p +
\frac12 \lambda^{p-2} \biint_Q
\frac{\big|\u^\frac{q+1}{2} -\a^\frac{q+1}{2}\big|^2}{\rho^{q+1}} \, \d x \d t\\\nonumber
&\quad+
c\eps^{-\beta} \left[ \biint_Q |Du|^{p-1} + |F|^{p-1} \, \d x \d t \right]^\frac{p}{p-1}
\end{align}
for every $\eps\in(0,1)$. This completes the estimate of $\mathrm{II}$
in the case~\eqref{eq:theta-intrinsic-q+1}. 
On the other hand, in the case~\eqref{eq:theta-degenerate} we have
$$
\theta^p
\leq
c \biint_Q |Du|^p + |F|^p \, \d x \d t
\le
c\lambda^p.
$$
In the last step we used~\eqref{eq:lambda-intrinsic}. 
Inserting this estimate into~\eqref{bound-II-theta}, we obtain
\begin{align*}
\mathrm{II} \leq c \lambda^\frac{q+1-p}{q} \left[ \biint_Q |Du|^{p-1} + |F|^{p-1} \, \d x \d t \right]^\frac{q+1}{q}.
\end{align*}
If $q+1>p$, we apply Young's inequality with exponents
$\frac{pq}{q+1-p}$ and $\frac{pq}{(p-1)(q+1)}$ and arrive at
\begin{align*}
  \mathrm{II} \leq \eps \lambda^p + c\eps^{-\beta} \left[ \biint_Q |Du|^{p-1} + |F|^{p-1} \, \d x \d t \right]^\frac{p}{p-1}.
\end{align*}
In the borderline case $q+1=p$, the same estimate is
immediate. Consequently, the bound~\eqref{bound-II-final} for $\mathrm{II}$ holds
true in every case considered in the lemma. Combining this with 
estimate~\eqref{bound-I-final} of $\mathrm{I}$ and recalling the
definition of $\mathrm{I}$ and $\mathrm{II}$ in~\eqref{def-I-II}, we
deduce
\begin{align*}
  &\lambda^{p-2} \biint_{Q}
  \frac{|\u^\frac{q+1}{2} - \a^\frac{q+1}{2}|^2}{\rho^{q+1}} \, \d x
    \d t \\\nonumber
  &\le
    \frac12 \lambda^{p-2} \biint_Q
\frac{\big|\u^\frac{q+1}{2} -\a^\frac{q+1}{2}\big|^2}{\rho^{q+1}} \,
    \d x \d t\\
   &\quad+
    c\eps \left(  \sup_{t\in \Lambda} \lambda^{p-2} \bint_B
    \frac{\big|\u^\frac{q+1}{2} - \a^\frac{q+1}{2}\big|^2}{\rho^{q+1}}
    \, \d x + \lambda^p+ \biint_Q |Du|^p \, \d x \d t \right) \\\nonumber
&\quad + c\eps^{-\beta} \left( \biint_Q |Du|^{\nu p} \, \d x \d t \right)^\frac{1}{\nu} + c\, \biint_Q |F|^p \, \d x \d t.
\end{align*}
We reabsorb the first term on the right-hand side into the
left-hand side and estimate the term $\lambda^p$ by the
$\lambda$-intrinsic scaling~\eqref{eq:lambda-intrinsic}. This yields
the asserted estimate after replacing $\eps$ by $\frac{\eps}{c}$. 
\end{proof}

Next, we give an auxiliary result that will be needed in the proof of
the second Sobolev-Poincar\'e inequality.

\begin{lemma} \label{lem:q+1-slicewise-estimate}
Let $q > 1$, $\frac{n(q+1)}{n+q+1} < p \leq q+1$ and
assume that $Q_{2\rho}^{(\lambda,\theta)}(z_o) \Subset \Omega_T$
and that the $\lambda$- and $\theta$-subintrinsic scaling
properties~\eqref{eq:lambda-intrinsic}$_1$
and~\eqref{eq:theta-intrinsic-q+1}$_1$ are satisfied. Then, there exists a constant
$c > 0$ depending on $n,p,q,C_\theta$ and $C_\lambda$ such that
for
every function $u\in L_{\mathrm{loc}}^p(0,T;W_{\mathrm{loc}}^{1,p}(\Omega,\R^N))\cap
L^\infty_{\mathrm{loc}}(0,T;L_{\mathrm{loc}}^{q+1}(\Omega,\R^N))$,
we have 
\begin{align*}
&\biint_{Q_\rho^{(\lambda,\theta)}(z_o)} \frac{\babs{u - \big[(\u^q)_{x_o;\hat\rho}^{(\theta)}\big]^{\frac{1}{q}}(t)}^{q+1}}{\rho^{q+1}} \, \d x \d t \\
&\ \ \leq c \bigg(\biint_{Q_\rho^{(\lambda,\theta)}(z_o)}|Du|^{\nu
                                                                                                                                                           p}\dx\dt\bigg)^{\frac{2(q+1)}{2(q+1)+\nu p(q-1)}} \\
&\qquad\cdot\left( \sup_{t\in \Lambda_\rho^{(\lambda)}(t_o)}\bint_{B_\rho^{(\theta)}(x_o)} \frac{|u-a|^{q+1}}{\rho^{q+1}} \, \d x \right)^{\frac{2(q+1-\nu p)}{2(q+1)+\nu p(q-1)}}    
\end{align*}
for every $\nu\in[\frac{n(q+1)}{p(n+q+1)},1]$, every
$\hat\rho\in[\frac\rho2,\rho]$ and every $a\in\R^N$. 
In particular, we have
\begin{align*}
\biint_{Q_\rho^{(\lambda,\theta)}(z_o)} &\frac{\babs{u - \big[(\u^q)_{x_o;\hat\rho}^{(\theta)}\big]^{\frac{1}{q}}(t)}^{q+1}}{\rho^{q+1}} \, \d x \d t\\
&\leq c \lambda^\frac{2(2(q+1)+ p(p-2))}{2(q+1) +p(q-1)} \left( \sup_{t\in \Lambda_\rho^{(\lambda)}(t_o)} \bint_{B_\rho^{(\theta)}(x_o)} \frac{|u-a|^{q+1}}{\lambda^{2-p}\rho^{q+1}} \, \d x \right)^\frac{2(q+1-p)}{2(q+1) + p(q-1)}.
\end{align*}
\end{lemma}

\begin{proof}
As in the preceding proof, we abbreviate $Q:=Q_\rho^{(\lambda,\theta)}(z_o)$,
$B:=B_\rho^{(\theta)}(x_o)$, $\widehat B:=B_{\hat\rho}^{(\theta)}(x_o)$ and
$\Lambda:=\Lambda_\rho^{(\lambda)}(t_o)$.
First, we apply Lemma~\ref{lem:uavetoa} with $\alpha=\frac1q$ and
$p=q+1$ to exchange the mean value of $\u^q$ by the mean value of $u$. 
Then, we note that the fact $\nu\ge\frac{n(q+1)}{p(n+q+1)}$ allows us
to use the
Gagliardo-Nirenberg inequality from Lemma~\ref{lem:GN} with the
parameters $(p,q,r,\vartheta)$ replaced by $(q+1,\nu p,q+1,\frac{\nu
  p}{q+1})$. Finally, we apply Poincar\'e's inequality slicewise. In this way, we obtain
\begin{align*}
  &\biint_{Q}\frac{\big|u-\big[(\u^q)_{\widehat
    B}\big]^{\frac1q}(t)\big|^{q+1}}{\rho^{q+1}} \, \d x \d t
    \le
  c\,\biint_{Q}\frac{|u-(u)_B(t)|^{q+1}}{\rho^{q+1}} \, \d x \d t\\
  &\quad\le
    c\theta^{1-q}
    \biint_Q \bigg[|D u|^{\nu p}+\frac{|u-(u)_B(t)|^{\nu p}}{\big(\theta^{\frac{1-q}{1+q}}\rho\big)^{\nu p}}\bigg] \d x \d t \\
    &\quad\phantom{=} \cdot 
    \Bigg(\sup_{t\in\Lambda}\bint_B\frac{|u-(u)_B(t)|^{q+1}}{\theta^{1-q}\rho^{q+1}}\dx\Bigg)^{1-\frac{\nu
    p}{q+1}}\\
  &\quad\le
    c\theta^{-\nu p\frac{q-1}{q+1}}
    \biint_Q |D u|^{\nu p} \d x \d t\ 
    \Bigg(\sup_{t\in\Lambda}\bint_B\frac{|u-a|^{q+1}}{\rho^{q+1}}\dx\Bigg)^{1-\frac{\nu
    p}{q+1}}.
\end{align*}
In the last step we applied Lemma~\ref{lem:uavetoa} again. 
We use assumption~\eqref{eq:theta-intrinsic-q+1}$_1$ in
order to bound the negative power of $\theta$ appearing on the
right-hand side from above. In this way, we obtain
\begin{align*}
  &\biint_{Q}\frac{\big|u-\big[(\u^q)_{\widehat B}\big]^{\frac1q}(t)\big|^{q+1}}{\rho^{q+1}}
  \, \d x \d t\\
  &\qquad\le
    c\bigg(\biint_{Q}\frac{\big|u-\big[(\u^q)_{\widehat B}\big]^{\frac1q}(t)\big|^{q+1}}{\rho^{q+1}}
    \, \d x \d t\bigg)^{-\frac{\nu p(q-1)}{2(q+1)}}\\
  &\qquad\qquad\cdot\biint_Q |D u|^{\nu p} \d x \d t\ 
    \Bigg(\sup_{t\in\Lambda}\bint_B\frac{|u-a|^{q+1}}{\rho^{q+1}}\dx\Bigg)^{\frac{q+1-\nu
    p}{q+1}}.
\end{align*}
By absorbing the first integral on the right-hand side into the left
and taking both sides to the power $\frac{2(q+1)}{2(q+1)+\nu p(q-1)}$,
we deduce the first asserted estimate. The second assertion follows by
choosing $\nu=1$ and using~\eqref{eq:lambda-intrinsic}$_1$. 
\end{proof}

Now we are in a position to prove a Sobolev-Poincar\'e inequality for
the first term on the right-hand side of the energy
estimate~\eqref{energy-estimate}.

\begin{lemma} \label{lem:sobo-poincare-p}
  Suppose that $q > 1$, $\frac{n(q+1)}{n+q+1}<p \leq q+1$,
  and that $u$ is a weak solution to~\eqref{eq:dne}, where 
  assumption~\eqref{assumption:A} is satisfied. Moreover, we consider a cylinder $Q_{2\rho}^{(\lambda,\theta)}(z_o) \Subset \Omega_T$ and assume that 
the $\lambda$-intrinsic coupling~\eqref{eq:lambda-intrinsic} and
additionally, property~\eqref{eq:theta-intrinsic-q+1} or
\eqref{eq:theta-degenerate} are satisfied. Then the following Sobolev-Poincar\'e inequality holds:
\begin{align*}
\theta^{p\frac{q-1}{q+1}} &\biint_{Q_\rho^{(\lambda,\theta)}(z_o)}\frac{| u - a |^p}{ \rho^{p} } \, \d x \d t \\
&\leq \eps \left( \sup_{t \in \Lambda_\rho^{(\lambda)}(t_o)} \lambda^{p-2} \bint_{B_\rho^{(\theta)}(x_o)} \frac{\babs{ \u^\frac{q+1}{2}(t)-\a^\frac{q+1}{2} }^2}{ \rho^{q+1} } \, \d x + \biint_{Q_\rho^{(\lambda,\theta)}(z_o)} |Du|^p \, \d x \d t \right) \\
&\phantom{+} + c \eps^{-\beta} \left[ \left( \biint_{Q_\rho^{(\lambda,\theta)}(z_o)} |Du|^{\nu p} \, \d x \d t  \right)^\frac{1}{\nu} + \biint_{Q_\rho^{(\lambda,\theta)}(z_o)} |F|^p \, \d x \d t  \right],
\end{align*}
where $\max \left\{\frac{n(q+1)}{p(n+q+1)} , \frac{p-1}{p},\frac{n}{n+2},\frac{n}{n+2}(1+\frac{2}{p}-\frac{2}{q}) \right\} \leq \nu \leq 1$ and $a = (u)_{z_o;\rho}^{(\theta,\lambda)}$. The preceding estimate holds for an arbitrary $\eps \in (0,1)$ with a constant $c = c(n,p,q,C_1,C_\theta,C_\lambda)> 0$ and $\beta = \beta(n,p,q)>0$.
\end{lemma}

\begin{proof}
We continue to use the notations $Q,\widehat Q,B,\widehat B$ and
$\Lambda$ introduced in the preceding proofs.   
We begin with two easy cases, in which the assertion can be deduced
from Lemma~\ref{lem:sobo-poincare-q+1}.

\emph{Case 1: The $\theta$-singular case \eqref{eq:theta-degenerate}.}
  In this case, assumptions \eqref{eq:theta-degenerate}
  and~\eqref{eq:lambda-intrinsic} imply $\theta\le
  c\lambda$. Moreover, we use H\"older's inequality,
  Lemma~\ref{lem:a-b} with $\alpha=\frac{q+1}{2}$, and finally,
  Young's inequality with exponents $\frac{q+1}{q+1-p}$ and
  $\frac{q+1}{p}$. In this way, we obtain the bound 
  \begin{align*}
    \theta^{p\frac{q-1}{q+1}} \biint_{Q}\frac{|u-a |^p}{\rho^{p}}\,\d
    x\d t
    &\le
      c\lambda^{p\frac{q-1}{q+1}} \bigg(\biint_{Q}\frac{|u-a|^{q+1}}{\rho^{q+1}}\,\d
      x\d t\bigg)^{\frac{p}{q+1}}\\
    &\le
      c\lambda^{p\frac{q+1-p}{q+1}} \bigg(\lambda^{p-2}\biint_{Q}\frac{\big|\u^{\frac{q+1}{2}}-\a^{\frac{q+1}{2}}\big|^{2}}{\rho^{q+1}}\,\d
      x\d t\bigg)^{\frac{p}{q+1}}\\
    &\le
      \eps\lambda^p+c\eps^{-\beta}\lambda^{p-2}\biint_{Q}\frac{\big|\u^{\frac{q+1}{2}}-\a^{\frac{q+1}{2}}\big|^{2}}{\rho^{q+1}}\,\d
      x\d t. 
  \end{align*}
  Again, we write $\beta$ for a positive universal constant that
  depends at most on $n,p$ and $q$.
  At this stage, the claim follows by estimating the last term with
  the help of Lemma~\ref{lem:sobo-poincare-q+1}.

  \emph{Case 2: The $\theta$-intrinsic case~\eqref{eq:theta-intrinsic-q+1} with
  $p\le 2$.} 
  As a consequence of~\eqref{eq:theta-intrinsic-q+1} we have 
$$
\theta \leq c \left( \biint_{Q} \frac{|u-a|^{q+1}}{\rho^{q+1}} \, \d x \d t \right)^\frac{1}{2} + c \frac{|a|^\frac{q+1}{2}}{\rho^\frac{q+1}{2}}.
$$
Using this together with H\"older's inequality, we infer 
\begin{align*}
&\theta^\frac{p(q-1)}{q+1} \biint_{Q}   \frac{|u-a|^p}{\rho^p} \, \d x \d t \\
&\leq c \left( \biint_{Q}   \frac{|u-a |^{q+1}}{\rho^{q+1}} \, \d x \d t \right)^\frac{p}{2} + c \left( \frac{|a|}{\rho} \right)^\frac{p(q-1)}{2} \biint_{Q}   \frac{|u-a |^p}{\rho^p} \, \d x \d t.
\end{align*}
We estimate the first term on the right-hand side by
Lemma~\ref{lem:a-b} with $\alpha=\frac{q+1}{2}$ and the second term by
Lemma~\ref{lem:Acerbi-Fusco} with the same value of $\alpha$. In this
way we get
\begin{align*}
  &\theta^\frac{p(q-1)}{q+1} \biint_{Q}   \frac{|u-a|^p}{\rho^p} \, \d x\d t\\
  &\qquad\le
    c \left( \biint_{Q}
    \frac{\big|\u^\frac{q+1}{2}-\a^\frac{q+1}{2}\big|^2}{\rho^{q+1}}
    \, \d x \d t \right)^\frac{p}{2} +
    c \biint_{Q}
    \frac{\big|\u^\frac{q+1}{2}-\a^\frac{q+1}{2}\big|^p}{\rho^{\frac{p(q+1)}{2}}} \, \d
    x \d t\\
  &\qquad
    \le
    c\lambda^{\frac{(2-p)p}{2}} \left( \lambda^{p-2}\biint_{Q}
    \frac{\big|\u^\frac{q+1}{2}-\a^\frac{q+1}{2}\big|^2}{\rho^{q+1}}
    \, \d x \d t \right)^\frac{p}{2}.
\end{align*}
The last estimate follows from H\"older's inequality, since $p\le2$. 
If $p < 2$ we may directly use Young's inequality with exponents
$\frac{2}{2-p}$ and $\frac2p$, which results in the estimate
\begin{align*}
  \theta^\frac{p(q-1)}{q+1} \biint_{Q}   \frac{|u-a|^p}{\rho^p} \, \d x\d t
  &\le
    \eps\lambda^p
    +
    c\eps^{-\beta}\lambda^{p-2}\biint_{Q}
    \frac{\big|\u^\frac{q+1}{2}-\a^\frac{q+1}{2}\big|^2}{\rho^{q+1}}\,\d x\d t
\end{align*}
for every $\eps\in(0,1)$. 
In the case $p=2$, this is an immediate consequence of the preceding inequality.
Now, the asserted estimate again follows by applying
Lemma~\ref{lem:sobo-poincare-q+1} to the last integral.

Now we turn our attention to the final case, which turns out to be
much more involved.

\emph{Case 3: The $\theta$-intrinsic case~\eqref{eq:theta-intrinsic-q+1} with $p>2$.}
By using triangle inequality and Lemma~\ref{lem:uavetoa} with $\alpha=1$, we write
\begin{align*}
\theta^{p\frac{q-1}{q+1}}
  \biint_{Q}\frac{| u - a |^p}{ \rho^{p}} \, \d x \d t
  &\leq c\theta^{p\frac{q-1}{q+1}}
        \biint_{Q}\frac{\big| u -
        \big[(\u^q)_{\widehat B}\big]^\frac{1}{q}(t)\big|^p}{ \rho^{p} } \, \d x \d t \\
  &\phantom{+}
    +c\frac{\theta^{2p\frac{q-1}{q+1}}}{\theta^{p\frac{q-1}{q+1}}}
    \bint_{\Lambda}\frac{\big|\big[(\u^q)_{\widehat B}\big]^\frac{1}{q}(t)
    - \big[(\u^q)_{\widehat Q}\big]^\frac{1}{q}\big|^p}{ \rho^{p} } \,  \d t \\
&=: \mathrm{I} + \mathrm{II}.
\end{align*}
The $\theta$-superintrinsic scaling~\eqref{eq:theta-intrinsic-q+1}$_2$ implies
\begin{align*}
  \theta^2
  \le
  c\bigg(\frac{|a|}{\rho}\bigg)^{q+1}
  +
  c\biint_{Q} \frac{|u-a|^{q+1}}{\rho^{q+1}} \, \d x \d t.
\end{align*}
We use this to estimate the term $\mathrm{I}$ and twice apply H\"older's
inequality in the space integral, denoting $\sigma =
\max\{p,q\}$. Afterwards,  we apply Lemma~\ref{lem:uavetoa}, once with
$\alpha=\frac1q$ and $p=\sigma$, and once with $\alpha=\frac1q$ and $p=q+1$. Note
that in particular the first application is possible since $\sigma\ge
q$. This procedure leads to the estimate
\begin{align*}
\mathrm I &\leq \left( \frac{|a|}{\rho} \right)^{p\frac{q-1}{2}} \bint_{\Lambda} \left( \bint_{B} \frac{|u - (u)_B (t)|^\sigma}{\rho^\sigma} \, \d x \right)^\frac{p}{\sigma} \, \d t \\
&\phantom{+} + \left(\biint_{Q} \frac{|u-a|^{q+1}}{\rho^{q+1}} \, \d x \d t\right)^{\frac{p}{2} \frac{q-1}{q+1}} \bint_{\Lambda} \left( \bint_{B} \frac{|u-(u)_B(t)|^{q+1}}{\rho^{q+1}} \, \d x \right)^\frac{p}{q+1} \, \d t \\
&=: \mathrm{I}_1  + \mathrm{I}_2.
\end{align*}
By using Lemma~\ref{lem:GN} with $(p,q,r,\vartheta)$ replaced by
$(\sigma,\nu p,2,\nu)$, which is possible since $\nu \ge \frac{n}{n+2}
\max \left\{ 1, 1+\frac{2}{p} - \frac{2}{q} \right\}$, we have
\begin{align}\label{first-bound-I1}
  \mathrm{I}_1
  \leq
  c\left(\frac{|a|}{\rho}\right)^{p\frac{q-1}{2}}
  \theta^{-p \frac{q-1}{q+1}}&\biint_{Q} \Bigg[|D u|^{\nu p} + \frac{ \babs{ u - (u)_{B}(t)}^{\nu p}}{\left( \theta^{\frac{1-q}{1+q}} \rho \right)^{\nu p}}\Bigg] \, \d x \d t \\\nonumber
&\phantom{+}\cdot\left(\sup_{t \in\Lambda}\bint_{B}\frac{\babs{u-(u)_{B}(t)}^{2}}{\left(\theta^{\frac{1-q}{1+q}}\rho\right)^{2}}\,\d x \right)^\frac{(1-\nu)p}{2}. 
\end{align}
In the next step, we use Poincar\'e's inequality slice-wise and
rearrange the terms. Then, we note that the
$\theta$-subintrinsic scaling~\eqref{eq:theta-intrinsic-q+1}$_1$
implies 
$(\frac{|a|}{\rho})^{q+1}\le c\theta^2$.
For the estimate of the $\sup$-term, we use Lemma~\ref{lem:uavetoa} with $\alpha=1$ and $p=2$, and
then Lemma~\ref{lem:Acerbi-Fusco} with the parameter
$\alpha=\frac{q+1}{2}$. This leads to the estimate 
\begin{align*}
\mathrm{I}_1 &\leq c \left( \frac{|a|}{\rho} \right)^{p \nu\frac{q-1}{2}} \!\theta^{-p\nu \frac{q-1}{q+1}} \biint_{Q} |D u|^{\nu p} \d x \d t \left( \sup_{t \in \Lambda} \bint_{B} \frac{ |a|^{q-1}\babs{ u - (u)_{B}(t)}^{2}}{ \rho^{q+1}}\d x \right)^\frac{(1-\nu)p}{2} \\
&\leq c\lambda^\frac{ (2-p)(1-\nu)p}{2} \biint_{Q} |Du|^{\nu p}  \d x \d t \left( \sup_{t \in \Lambda} \bint_{B} \frac{ \babs{ \u^\frac{q+1}{2} - \a^\frac{q+1}{2}}^{2}}{ \lambda^{2-p} \rho^{q+1}}  \d x \right)^\frac{(1-\nu )p}{2}.
\end{align*}
  Since $\nu\ge\frac{p-1}{p}$, we may use 
  Young's inequality with exponents $\frac{2}{(1-\nu)p}$ and
  $\frac{2}{2 - (1-\nu)p}$ to get 
\begin{align*}
\mathrm{I}_1 &\leq \eps \sup_{t \in \Lambda} \bint_{B} \frac{ \babs{ \u^\frac{q+1}{2} - \a^\frac{q+1}{2}}^{2}}{ \lambda^{2-p} \rho^{q+1}} \, \d x +  c\eps^{-\beta} \left( \lambda^\frac{(2-p)(1-\nu)p}{2} \biint_{Q} |D u|^{\nu p} \, \d x \d t \right)^\frac{2}{2-(1-\nu)p}.
\end{align*}
By using the $\lambda$-subintrinsic
scaling~\eqref{eq:lambda-intrinsic}$_1$, which implies
\begin{equation}\label{lambda-subintrinsic}
  \lambda \geq c\left( \biint_{Q} |D u|^{\nu p} \, \d x \d t
  \right)^\frac{1}{\nu p},
\end{equation}
together with the fact $p>2$, we arrive at the estimate
\begin{align}\label{bound-I-1}
  \mathrm{I}_1
  \le
  \eps \sup_{t \in \Lambda} \bint_{B} \frac{ \babs{ \u^\frac{q+1}{2} - \a^\frac{q+1}{2}}^{2}}{ \lambda^{2-p} \rho^{q+1}} \, \d x +  c\eps^{-\beta} \left( \biint_{Q} |D u|^{\nu p} \, \d x \d t \right)^\frac{1}{\nu}.
\end{align}
Next, we estimate the term $\mathrm{I}_2$. Since
$p > \frac{n(q+1)}{n+q+1}$, the Sobolev-Poincar\'e inequality implies
\begin{align}\label{I2-sobolev}
  &\bint_{\Lambda} \left( \bint_{B}\frac{|u-(u)_B(t)|^{q+1}}{\rho^{q+1}} \,\d x\right)^\frac{p}{q+1} \, \d t\\\nonumber
  &\qquad\le
    c\,\theta^{-p\frac{q-1}{q+1}}\biint_Q |Du|^p\dx\dt
    \le
    c\,\theta^{-p\frac{q-1}{q+1}}\lambda^p.
\end{align}
In the last step, we used \eqref{eq:lambda-intrinsic}. Furthermore,
since $Q$ is $\theta$-subintrinsic in the sense of~\eqref{eq:theta-intrinsic-q+1}$_1$, we have 
\begin{align*}
  &\bint_{\Lambda} \left( \bint_{B}\frac{|u-(u)_B(t)|^{q+1}}{\rho^{q+1}} \, \d x
  \right)^\frac{p}{q+1} \, \d t\\
  &\qquad\le
  c\bigg(\biint_Q\frac{|u|^{q+1}}{\rho^{q+1}}\dx\dt\bigg)^{\frac{p}{q+1}\frac{q-1}{q+1}}
  \bigg(\bint_{\Lambda} \left( \bint_{B}\frac{|u-(u)_B(t)|^{q+1}}{\rho^{q+1}} \, \d x
  \right)^\frac{p}{q+1} \, \d t\bigg)^{\frac{2}{q+1}}\\
  &\qquad\le
  c\theta^{\frac{2p}{q+1}\frac{q-1}{q+1}}
  \bigg(\bint_{\Lambda} \left( \bint_{B}\frac{|u-(u)_B(t)|^{q+1}}{\rho^{q+1}} \, \d x
  \right)^\frac{p}{q+1} \, \d t\bigg)^{\frac{2}{q+1}}.
\end{align*}
Estimating the right-hand side by~\eqref{I2-sobolev}, we observe that
the powers of $\theta$ cancel each other out. Therefore, we obtain the bound 
\begin{equation}\label{I2-lambda}
  \bint_{\Lambda} \left( \bint_{B}\frac{|u-(u)_B(t)|^{q+1}}{\rho^{q+1}} \, \d x
    \right)^\frac{p}{q+1} \, \d t
  \le
  c\lambda^{\frac{2p}{q+1}}.
\end{equation}
In order to estimate $\mathrm{I}_2$, we apply the triangle inequality
and use~\eqref{I2-lambda} in the first of the resulting terms and
\eqref{I2-sobolev} in the second. This leads to the bound 
\begin{align*}
  \mathrm{I}_2
  &\le
    c\left(\biint_{Q}
    \frac{\big|u-\big[(\u^q)_{\widehat B}\big]^{\frac1q}(t)\big|^{q+1}}{\rho^{q+1}} \, \d x \d
    t\right)^{\frac{p}{2} \frac{q-1}{q+1}}
    \lambda^{\frac{2p}{q+1}} \\
  &\quad+
    c\left(\biint_{Q}
    \frac{\big|\big[(\u^q)_{\widehat
    B}\big]^{\frac1q}(t)-\big[(\u^q)_{\widehat Q}\big]^\frac{1}{q}\big|^{q+1}}{\rho^{q+1}} \, \d x \d
    t\right)^{\frac{p}{2} \frac{q-1}{q+1}}
    \theta^{-p\frac{q-1}{q+1}}\lambda^p\\
  &=:
    \mathrm{I}_{2,1}+\mathrm{I}_{2,2}.
\end{align*}
For the estimate of the first term, we use Young's inequality with
exponents $\frac{q+1}{q-1}$ and $\frac{q+1}{2}$ and then
Lemma~\ref{lem:q+1-slicewise-estimate}, which yields the bound
\begin{align*}
  \mathrm{I}_{2,1}
  &\le
    \eps\lambda^p+c\eps^{-\beta} \left(\biint_{Q}
    \frac{\big|u-\big[(\u^q)_{\widehat B}\big]^{\frac1q}(t)\big|^{q+1}}{\rho^{q+1}} \, \d x \d
    t\right)^{\frac{p}{2}}\\
  &\le
    \eps\lambda^p +
    c\eps^{-\beta} \left( \sup_{t\in\Lambda}\bint_{B}
    \frac{|u-a|^{q+1}}{\lambda^{2-p}\rho^{q+1}} \, \d x \right)^{\frac{p(q+1-\nu p)}{2(q+1)+\nu p(q-1)}}\\
  &\qquad\qquad\quad\cdot\lambda^{(2-p)\frac{p(q+1-\nu p)}{2(q+1)+\nu
    p(q-1)}}\bigg(\biint_{Q}|Du|^{\nu p}\dx\dt\bigg)^{\frac{p(q+1)}{2(q+1)+\nu p(q-1)}} .
\end{align*}
Since $2<p\le q+1$, the power of $\lambda$ in the last line is
negative. Therefore, we can use the $\lambda$-subintrinsic scaling~\eqref{eq:lambda-intrinsic}$_1$
in the form of~\eqref{lambda-subintrinsic} to estimate the power of $\lambda$ from above. This leads to the bound
\begin{align*}
  \mathrm{I}_{2,1}
  &\le
    \eps\lambda^p
    +c\eps^{-\beta} \left( \sup_{t\in \Lambda}\bint_{B} \frac{|u-a|^{q+1}}{\lambda^{2-p}\rho^{q+1}} \, \d x \right)^{\frac{p(q+1-\nu p)}{2(q+1)+\nu p(q-1)}}\\
  &\qquad\qquad\quad \cdot\Bigg(\bigg[\biint_{Q}|Du|^{\nu p}\dx\dt\bigg]^{\frac{1}{\nu}}\Bigg)^\frac{2(q+1)+\nu
    p(q-1)-p(q+1-\nu p)}{2(q+1)+\nu p(q-1)}.
\end{align*}
Since $\nu p\ge p-1>p-2$, the exponent of the $\sup$-term is smaller than
one, and it is positive. Moreover, both exponents outside the round brackets add up to one. Therefore, another application of Young's inequality yields
\begin{align}\label{bound-I-21}
  \mathrm{I}_{2,1}
  \le
    \eps\lambda^p
    +
    \eps\sup_{t\in \Lambda}\bint_{B}
    \frac{|u-a|^{q+1}}{\lambda^{2-p}\rho^{q+1}} \, \d x
    +
    c\eps^{-\beta} \bigg(\biint_{Q}|Du|^{\nu p}\dx\dt\bigg)^{\frac1\nu}.
\end{align}
For the estimate of $\mathrm{I}_{2,2}$, we use Lemma~\ref{lem:a-b}
with $\alpha=q$ and then Lemma~\ref{lem:gluing}, which implies
\begin{align}\label{first-bound-I-22}
  \mathrm{I}_{2,2}
  &\le c\Bigg(\bint_{\Lambda}
    \frac{\big| {(\u^q)_{\widehat B}}(t) -
              (\u^q)_{\widehat Q}\big|^{\frac{q+1}{q}}}{\rho^{q+1}} \, \d t\Bigg)^{\frac{p}{2}\frac{q-1}{q+1}}
    \theta^{-p\frac{q-1}{q+1}}\lambda^p\\\nonumber
  &\le c\Bigg(\bint_{\Lambda} \bint_{\Lambda} \frac{\big|
    {(\u^q)_{\hat \rho}^{(\theta)}}(t) - (\u^q)_{\hat \rho}^{(\theta)}(\tau)\big|^\frac{q+1}{q}}{\rho^{q+1}} \, \d t \d \tau \Bigg)^{\frac{p}{2}\frac{q-1}{q+1}}
    \theta^{-p\frac{q-1}{q+1}}\lambda^p\\\nonumber
  &\le
    c\bigg(\lambda^{2-p}\theta^{\frac{q-1}{q+1}}\biint_Q|Du|^{p-1}+|F|^{p-1}\dx\dt\bigg)^{\frac{p}{2}\frac{q-1}{q}}
    \theta^{-p\frac{q-1}{q+1}}\lambda^p\\\nonumber
  &=
    c\theta^{-p\frac{q-1}{2q}}\lambda^{p\frac{2q+(2-p)(q-1)}{2q}}\bigg(\biint_Q|Du|^{p-1}+|F|^{p-1}\dx\dt\bigg)^{\frac{p(q-1)}{2q}}.
\end{align}
Note that we can assume
\begin{align*}
  \biint_Q|Du|^{p-1}+|F|^{p-1}\dx\dt\le\theta^{p-1}
\end{align*}
since otherwise, the assertion of the lemma clearly holds,
because~\eqref{eq:theta-intrinsic-q+1}$_1$ implies that the left-hand side of the
asserted estimate is bounded by $c\theta^p$. Using this observation in order to
bound the negative powers of $\theta$ in the preceding estimate, we
arrive at
\begin{align*}
  \mathrm{I}_{2,2}
  &\le
    c\lambda^{p\frac{2q+(2-p)(q-1)}{2q}}\bigg(\biint_Q|Du|^{p-1}+|F|^{p-1}\dx\dt\bigg)^{\frac{p(q-1)}{2q}\frac{p-2}{p-1}}.
\end{align*}
In case $2q + (2-p)(q-1) < 0$, we use the $\lambda$-subintrinsic scaling~\eqref{eq:lambda-intrinsic}$_1$ and obtain
$$
\mathrm{I}_{2,2} \leq c\bigg(\biint_Q|Du|^{p-1}+|F|^{p-1}\dx\dt\bigg)^{\frac{p}{p-1}}.
$$
If $2q + (2-p)(q-1) = 0$, this estimate is identical to the preceding
one. 
In the remaining case, by observing that $\frac{2q + (2-p)(q-1)}{2q} < 1$, we use Young's inequality with exponents $\frac{2q}{2q + (2-p)(q-1)}$ and $\frac{2q}{(p-2)(q-1)}$ to obtain
$$
\mathrm{I}_{2,2} \leq \eps \lambda^p + c\eps^{-\beta} \bigg(\biint_Q|Du|^{p-1}+|F|^{p-1}\dx\dt\bigg)^{\frac{p}{p-1}},
$$
completing the treatment of the term $\mathrm{I}_{2,2}$. Combining
this result with~\eqref{bound-I-1} and \eqref{bound-I-21}, using
H\"older's inequality and Lemma~\ref{lem:a-b}, we infer the bound
\begin{align}\label{final-bound-I}
  \mathrm{I}
  &\le
  \eps\lambda^p+c\eps \sup_{t \in \Lambda} \bint_{B} \frac{ \babs{
  \u^\frac{q+1}{2} - \a^\frac{q+1}{2}}^{2}}{ \lambda^{2-p} \rho^{q+1}}
  \, \d x\\\nonumber
  &\qquad+  c\eps^{-\beta} \left( \biint_{Q} |D u|^{\nu p} \, \d x \d t
  \right)^\frac{1}{\nu}
  +c\eps^{-\beta}\biint_Q|F|^p\dx\dt.
\end{align}
By the $\theta$-superintrinsic
scaling~\eqref{eq:theta-intrinsic-q+1}$_2$ we have
\begin{align*}
  \theta^2
  &\le
  c\left( \frac{|\hat a|}{\rho} \right)^{q+1}
  +
  c \biint_{Q} \frac{\big|u - \big[(\u^q)_{\widehat B}\big]^\frac{1}{q}(t)\big|^{q+1}}{\rho^{q+1}} \, \d x \d t\\ 
  &\qquad+
  c \mint_{\Lambda} \frac{\big|\big[(\u^q)_{\widehat B}\big]^\frac{1}{q}(t) - \hat a\big|^{q+1}}{\rho^{q+1}} \, \d t,
\end{align*}
where we abbreviated $\hat a = [(\u^q)_{\widehat
  Q}]^\frac{1}{q}$. Using this for the estimate of $\mathrm{II}$, we
obtain 
\begin{align*}
\mathrm{II} 
&\leq c \theta^{-p \frac{q-1}{q+1}}\left( \frac{|\hat a|}{\rho} \right)^{(q-1)p}\bint_{\Lambda}\frac{\big| \big[(\u^q)_{\widehat B}\big]^\frac{1}{q}(t) - \hat a \big|^p}{ \rho^{p} } \,  \d t \\
&\phantom{+} +c\theta^{-p \frac{q-1}{q+1}}\left( \biint_{Q} \frac{\big|u - \big[(\u^q)_{\widehat B}\big]^\frac{1}{q}(t)\big|^{q+1}}{\rho^{q+1}} \, \d x \d t \right)^{p \frac{q-1}{q+1}}\bint_{\Lambda}\frac{\big| \big[(\u^q)_{\widehat B}\big]^\frac{1}{q}(t) - \hat a \big|^p}{ \rho^{p} } \,  \d t \\
&\phantom{+} +c\theta^{-p \frac{q-1}{q+1}}\left( \bint_{\Lambda} \frac{\big|\big[(\u^q)_{\widehat B}\big]^\frac{1}{q}(t) - \hat a\big|^{q+1}}{\rho^{q+1}} \, \d t \right)^{p \frac{q-1}{q+1}}\bint_{\Lambda}\frac{\big| \big[(\u^q)_{\widehat B}\big]^\frac{1}{q}(t) - \hat a \big|^p}{ \rho^{p} } \,  \d t \\
&=: \mathrm{II}_1 + \mathrm{II}_2 + \mathrm{II}_3.
\end{align*}
For the first term, we use in turn Lemma~\ref{lem:Acerbi-Fusco} with
$\alpha=q$, the gluing lemma (Lemma~\ref{lem:gluing}), the
$\lambda$-subintrinsic scaling~\eqref{eq:lambda-intrinsic}$_1$ and
then H\"older's inequality to get
\begin{align}\label{first-bound-II-1}
\mathrm{II}_1 &\leq c \theta^{- p \frac{q-1}{q+1}} \bint_{\Lambda}
                \frac{|(\u^q)_{\widehat B}(t) - \hat \a^q|^p}{\rho^{qp}} \, \d t
  \\\nonumber&\le
  c \theta^{- p \frac{q-1}{q+1}} \bint_{\Lambda} \bint_{\Lambda} \frac{\big|
    {(\u^q)_{\hat \rho}^{(\theta)}}(t) - (\u^q)_{\hat \rho}^{(\theta)}(\tau)\big|^{p}}{\rho^{pq}} \, \d t \d \tau\\\nonumber
&\leq c \lambda^{p(2-p)} \left( \biint_{Q} |Du|^{p-1} + |F|^{p-1} \, \d x \d t  \right)^p \\\nonumber
&\leq c \left( \biint_{Q} |Du|^{p-1} + |F|^{p-1} \, \d x \d t  \right)^\frac{p}{p-1} \\\nonumber
&\leq c \left( \biint_{Q} |Du|^{\nu p} \, \d x \d t \right)^\frac{1}{\nu} + c \biint_{Q} |F|^p \, \d x \d t . 
\end{align}
For the term $\mathrm{II}_3$ we use Lemma~\ref{lem:a-b} with
$\alpha=q$ and then H\"older's inequality to estimate 
\begin{align*}
  \mathrm{II}_3
  &\leq c \theta^{- p \frac{q-1}{q+1}} \bigg(\bint_{\Lambda}
    \frac{|(\u^q)_{\widehat B}(t) - \hat \a^q|^{\frac{q+1}{q}}}{\rho^{q+1}} \, \d
    t\bigg)^{p\frac{q-1}{q+1}}\bint_{\Lambda} \frac{|(\u^q)_{\widehat B}(t) -
    \hat \a^q|^{\frac{p}{q}}}{\rho^{p}} \, \d t\\
 &\leq c \theta^{- p \frac{q-1}{q+1}} \bint_{\Lambda}
                \frac{|(\u^q)_{\widehat B}(t) - \hat \a^q|^p}{\rho^{qp}} \, \d t,
\end{align*}
by using also the fact $\frac{q+1}{q}\le 2<p$. Now we proceed
exactly as for the estimate of $\mathrm{II}_1$ and arrive at the bound
\begin{align*}
  \mathrm{II}_3
  \leq c \left( \biint_{Q} |Du|^{\nu p} \, \d x \d t \right)^\frac{1}{\nu} + c \biint_{Q} |F|^p \, \d x \d t.
\end{align*}
For the term $\mathrm{II}_2$, we divide the power of the second term
as $p \frac{q-1}{q+1} = \frac{p(q-1)^2}{2q(q+1)}+\frac{p(q-1)}{2q}$ and
estimate the first part using the $\theta$-subintrinsic
scaling~\eqref{eq:theta-intrinsic-q+1}$_1$. For the last integral in $\mathrm{II}_2$, we
apply Lemma~\ref{lem:a-b} with $\alpha=q$. The
resulting integrals are then estimated by Lemma~\ref{lem:q+1-slicewise-estimate} 
and Lemma~\ref{lem:gluing}, respectively.  This yields
 \begin{align*}
   \mathrm{II}_2
   &\le
   c\theta^{-\frac{p(q-1)}{q(q+1)}}\left( \biint_{Q} \frac{\big|u -\big[(\u^q)_{\widehat B}\big]^{\frac{1}{q}}(t)\big|^{q+1}}{\rho^{q+1}} \,
      \d x \d t \right)^{\frac{p(q-1)}{2q}}\bint_{\Lambda}\frac{|(\u^q)_{\widehat B}(t) - \hat \a^q |^{\frac{p}{q}}}{ \rho^{p} } \,\d t\\
  &\leq c \theta^{-\frac{p(q-1)}{q(q+1)}} \left( \lambda^\frac{2(2(q+1)+ p(p-2))}{2(q+1) +p(q-1)} \left( \sup_{t\in\Lambda} \bint_{B} \frac{|u-a|^{q+1}}{\lambda^{2-p}\rho^{q+1}} \, \d x \right)^\frac{2(q+1-p)}{2(q+1) + p(q-1)} \right)^\frac{p(q-1)}{2q} \\
&\phantom{+} \cdot \left( \lambda^{2-p} \theta^\frac{q-1}{q+1} \biint_Q |Du|^{p-1} + |F|^{p-1} \, \d x \d t \right)^\frac{p}{q}.
 \end{align*}
Observe that $\theta$ will cancel out on the right-hand side. Subsequently, we use Young's inequality with exponents $q$ and $\frac{q}{q-1}$ and obtain
\begin{align*}
\mathrm{II}_2 &\leq \eps \lambda^{p \frac{2(q+1) + p(p-2)}{2(q+1)+p(q-1)}} \left( \sup_{t\in\Lambda} \bint_{B} \frac{|u-a|^{q+1}}{\lambda^{2-p}\rho^{q+1}} \, \d x \right)^\frac{p(q+1-p)}{2(q+1) + p(q-1)} \\
&\phantom{+} + c\eps^{-\beta} \lambda^{p(2-p)} \left( \biint_Q |Du|^{p-1} + |F|^{p-1} \, \d x \d t \right)^p.
\end{align*}
For the first term we use Young's inequality with exponents
$\frac{2(q+1)+p(q-1)}{2(q+1) + p(p-2)}$ and $\frac{2(q+1) +
  p(q-1)}{p(q+1-p)}$ (observe that these exponents are $> 1$ in case
$2 < p < q+1$). For the last term, we use the $\lambda$-subintrinsic
scaling~\eqref{eq:lambda-intrinsic}$_1$ and the fact $p>2$ to deduce 
\begin{align*}
  \mathrm{II}_2
  &\leq
    \eps \lambda^p + \eps \sup_{t\in\Lambda} \bint_{B}\frac{|u-a|^{q+1}}{\lambda^{2-p}\rho^{q+1}} \,\d x
    +
    c\eps^{-\beta} \left( \biint_Q |Du|^{p-1} + |F|^{p-1} \, \d x \d t \right)^{\frac{p}{p-1}}.
\end{align*}
Collecting the estimates and applying H\"older's inequality and Lemma~\ref{lem:a-b}, we arrive at the bound
\begin{align*}
  \mathrm{II}
  &\le
    \eps \lambda^p + \eps \sup_{t\in\Lambda}
    \bint_{B}\frac{\big|\u^{\frac{q+1}{2}}-\a^{\frac{q+1}{2}}\big|^{2}}{\lambda^{2-p}\rho^{q+1}} \,\d x\\
  &\quad  +
    c\eps^{-\beta} \left( \biint_{Q} |Du|^{\nu p} \, \d x \d t \right)^\frac{1}{\nu} + c\eps^{-\beta} \biint_{Q} |F|^p \, \d x \d t.
\end{align*}
As stated in~\eqref{final-bound-I}, the term $\mathrm{I}$ is bounded
by exactly the same quantities. Therefore, the asserted estimate
follows by bounding $\lambda^p$ by means of the $\lambda$-intrinsic
scaling~\eqref{eq:lambda-intrinsic}. 
\end{proof}

\section{Parabolic Sobolev-Poincar\'e type inequalities in case \texorpdfstring{$q+1 < p$}{q+1<p}}
\label{sec:sobolev-slow-diffusion}

In this section, we prove versions of the Sobolev-Poincar\'e type
inequalities from the preceding section for the missing case $q+1<p$.
In this case, the $\theta$-intrinsic scaling~\eqref{eq:theta-intrinsic}
reads as 
\begin{equation} \label{eq:theta-intrinsic-p}
\frac{1}{C_\theta} \biint_{Q_{2\rho}^{(\lambda,\theta)}}  \frac{|u|^{p}}{(2\rho)^{p}} \, \d x \d t \leq \theta^\frac{2p}{q+1} \leq C_\theta \biint_{Q_{\rho}^{(\lambda,\theta)}}  \frac{|u|^{p}}{\rho^{p}} \, \d x \d t
\end{equation}
and the $\theta$-singular scaling~\eqref{eq:theta-degenerate-1} becomes
\begin{equation} \label{eq:theta-degenerate-p}
\frac{1}{C_\theta} \biint_{Q_{2\rho}^{(\lambda,\theta)}}  \frac{|u|^{p}}{(2\rho)^{p}} \, \d x \d t \leq \theta^\frac{2p}{q+1} \leq C_\theta \left( \biint_{Q_{\rho}^{(\lambda,\theta)}} |Du|^p + |F|^p \, \d x \d t \right)^\frac{2}{q+1}.
\end{equation}

We start with an auxiliary estimate that will be needed for the
estimate of the first Sobolev-Poincar\'e inequality.

\begin{lemma} \label{lem:p-slicewise-estimate}
Let $p> q+1 > 2$ and
assume that $Q_{2\rho}^{(\lambda,\theta)}(z_o) \Subset \Omega_T$
and that the $\lambda$- and $\theta$-subintrinsic scaling
properties~\eqref{eq:lambda-intrinsic}$_1$
and~\eqref{eq:theta-intrinsic-p}$_1$ are satisfied. Then, there exists a constant
$c > 0$ depending on $n,p,q,C_\theta$ and $C_\lambda$ such that
for
every function $u\in L_{\mathrm{loc}}^p(0,T;W_{\mathrm{loc}}^{1,p}(\Omega,\R^N))\cap
L^\infty_{\mathrm{loc}}(0,T;L_{\mathrm{loc}}^{q+1}(\Omega,\R^N))$,
we have 
\begin{align*}
&\biint_{Q_\rho^{(\lambda,\theta)}(z_o)} \frac{\babs{u - (u)_{x_o;\rho}^{(\theta)}(t)}^{p}}{\rho^{p}} \, \d x \d t \\
&\ \ \leq c \bigg(\biint_{Q_\rho^{(\lambda,\theta)}(z_o)}|Du|^{\nu
                                                                                                                                                           p}\dx\dt\bigg)^{\frac{2}{2+\nu (q-1)}} \\
&\qquad\cdot\left( \sup_{t\in \Lambda_\rho^{(\lambda)}(t_o)}\bint_{B_\rho^{(\theta)}(x_o)} \frac{|u-a|^{q+1}}{\rho^{q+1}} \, \d x \right)^{\frac{2p(1-\nu)}{(q+1)(2+\nu (q-1))}}    
\end{align*}
for every $\nu\in[\frac{n}{n+q+1},1]$ and every $a\in\R^N$. 
In particular, we have
\begin{align*}
\biint_{Q_\rho^{(\lambda,\theta)}(z_o)} &\frac{\babs{u - (u)_{x_o;\rho}^{(\theta)}(t)}^{p}}{\rho^{p}} \, \d x \d t \leq c \lambda^\frac{2p}{q+1}.
\end{align*}
\end{lemma}

\begin{proof}
As in the preceding section, we abbreviate $Q:=Q_\rho^{(\lambda,\theta)}(z_o)$,
$B:=B_\rho^{(\theta)}(x_o)$, $\widehat B:=B_{\hat\rho}^{(\theta)}(x_o)$ and
$\Lambda:=\Lambda_\rho^{(\lambda)}(t_o)$.
We note that the fact $\nu\ge\frac{n}{n+q+1}$ allows us
to use the
Gagliardo-Nirenberg inequality from Lemma~\ref{lem:GN} with the
parameters $(p,q,r,\vartheta)$ replaced by $(p,\nu p,q+1,\nu)$. Finally, we apply Poincar\'e's inequality slicewise. In this way, we obtain
\begin{align*}
  &\biint_{Q}\frac{|u-(u)_B(t)|^{p}}{\rho^{p}} \, \d x \d t\\
  &\quad\le
    c\theta^{-p\frac{q-1}{q+1}}
    \biint_Q \bigg[|D u|^{\nu p}+\frac{|u-(u)_B(t)|^{\nu p}}{\big(\theta^{\frac{1-q}{1+q}}\rho\big)^{\nu p}}\bigg] \d x \d t\\
    &\quad\phantom{=} \cdot
    \Bigg(\sup_{t\in\Lambda}\bint_B\frac{|u-(u)_B(t)|^{q+1}}{\theta^{1-q}\rho^{q+1}}\dx\Bigg)^{\frac{(1-\nu)
    p}{q+1}}\\
  &\quad\le
    c\theta^{-\nu p\frac{q-1}{q+1}}
    \biint_Q |D u|^{\nu p} \d x \d t\ 
    \Bigg(\sup_{t\in\Lambda}\bint_B\frac{|u-a|^{q+1}}{\rho^{q+1}}\dx\Bigg)^{\frac{(1-\nu)
    p}{q+1}}.
\end{align*}
In the last step we applied Lemma~\ref{lem:uavetoa}. 
We use assumption~\eqref{eq:theta-intrinsic-p}$_1$ in
order to bound the negative power of $\theta$ appearing on the
right-hand side from above. In this way, we obtain
\begin{align*}
  &\biint_{Q}\frac{|u-(u)_B(t)|^{p}}{\rho^{p}}
  \, \d x \d t\\
  &\qquad\le
    c\bigg(\biint_{Q}\frac{|u-(u)_B(t)|^{p}}{\rho^{p}}
    \, \d x \d t\bigg)^{-\frac{\nu (q-1)}{2}}\\
  &\qquad\qquad\cdot\biint_Q |D u|^{\nu p} \d x \d t\ 
    \Bigg(\sup_{t\in\Lambda}\bint_B\frac{|u-a|^{q+1}}{\rho^{q+1}}\dx\Bigg)^{{\frac{(1-\nu)
    p}{q+1}}}.
\end{align*}
By absorbing the first integral on the right-hand side into the left
and taking both sides to the power $\frac{2}{2+\nu (q-1)}$,
we deduce the first asserted estimate. The second assertion follows by
choosing $\nu=1$ and using~\eqref{eq:lambda-intrinsic}$_1$. 
\end{proof}

Next, we prove a Sobolev-Poincar\'e type inequality for the first term on the right-hand side of
the energy estimate~\eqref{energy-estimate}.

\begin{lemma} \label{lem:sobo-poincare-p-2}
Suppose that $p >q+1 > 2$
  and that $u$ is a weak solution to~\eqref{eq:dne}, under
  assumption~\eqref{assumption:A}. Moreover, we consider a cylinder
  $Q_{2\rho}^{(\lambda,\theta)}(z_o) \Subset \Omega_T$ and assume that 
the $\lambda$-intrinsic coupling~\eqref{eq:lambda-intrinsic} and
additionally, property~\eqref{eq:theta-intrinsic-p} or
\eqref{eq:theta-degenerate-p} are satisfied. Then the following Sobolev-Poincar\'e inequality holds:
\begin{align*}
\theta^{p\frac{q-1}{q+1}} &\biint_{Q_\rho^{(\lambda,\theta)}(z_o)}\frac{| u - a |^p}{ \rho^{p} } \, \d x \d t \\
&\leq \eps \left( \sup_{t \in \Lambda_\rho^{(\lambda)}(t_o)} \lambda^{p-2} \bint_{B_\rho^{(\theta)}(x_o)} \frac{\babs{ \u^\frac{q+1}{2}(t)-\a^\frac{q+1}{2} }^2}{ \rho^{q+1} } \, \d x + \biint_{Q_\rho^{(\lambda,\theta)}(z_o)} |Du|^p \, \d x \d t \right) \\
&\phantom{+} + c \eps^{-\beta} \left[ \left( \biint_{Q_\rho^{(\lambda,\theta)}(z_o)} |Du|^{\nu p} \, \d x \d t  \right)^\frac{1}{\nu} + \biint_{Q_\rho^{(\lambda,\theta)}(z_o)} |F|^p \, \d x \d t  \right],
\end{align*}
where $\max \left\{ \frac{p-1}{p},\frac{n}{n+2} \right\} \leq \nu \leq
1$ and $a = (u)_{z_o;\rho}^{(\theta,\lambda)}$. The preceding estimate
holds for any $\eps \in (0,1)$ with a constant $c =
c(n,p,q,C_1,C_\theta,C_\lambda)> 0$ and $\beta = \beta(n,p,q)>0$.
\end{lemma}

\begin{proof}
We continue to use the notations $Q,\widehat Q,B,\widehat B$ and
$\Lambda$ introduced in the preceding proofs.   
First observe that $p > q+1$ implies $p > 2$.
We distinguish between the cases \eqref{eq:theta-degenerate-p} and
\eqref{eq:theta-intrinsic-p}. 

\emph{Case 1: The $\theta$-singular case \eqref{eq:theta-degenerate-p}.}
  We use Lemma~\ref{lem:uavetoa} and the triangle inequality to estimate 
  \begin{align*}
    \theta^{p\frac{q-1}{q+1}} \biint_{Q}\frac{|u-a |^p}{\rho^{p}}\,\d
    x\d t
    &\le
    c \theta^{p\frac{q-1}{q+1}} \biint_{Q}\frac{\big|u-\big[(\u^q)_{\widehat B}\big]^\frac{1}{q}(t) \big|^p}{\rho^{p}}\,\d
    x\d t \\
    &\phantom{+} + c \theta^{p\frac{q-1}{q+1}}
      \biint_{Q}\frac{\big|\big[(\u^q)_{\widehat B}\big]^\frac{1}{q}(t) - \hat a\big|^p}{\rho^{p}}\,\d x\d t ,
  \end{align*}
  with $\hat a = [(\u^q)_{\widehat Q}]^\frac{1}{q}$.
  For the first term we use Lemmas~\ref{lem:uavetoa} and~\ref{lem:GN} with $(p,q,r,\vartheta)=(p,\nu p, q+1,\nu)$ to obtain 
\begin{align*}
\theta^{p\frac{q-1}{q+1}} \biint_{Q}\frac{\big|u-\big[(\u^q)_{\widehat B}\big]^\frac{1}{q}(t) \big|^p}{\rho^{p}}\,\d
    x\d t &\leq c \theta^\frac{p(1-\nu)(q-1)}{q+1} \lambda^\frac{p(2-p)(1-\nu)}{q+1} \biint_{Q} |Du|^{\nu p} \,\d
    x\d t \\
    &\qquad \cdot \left( \sup_{t\in\Lambda} \mint_B \frac{|u-a|^{q+1}}{\lambda^{2-p} \rho^p}\, \d x \right)^\frac{(1-\nu)p}{q+1}.
\end{align*}
Observe that $\nu \geq \frac{n}{n+2} > \frac{n}{n+q+1}$ such that Lemma~\ref{lem:GN} is applicable. 
Now we use \eqref{eq:theta-degenerate-p}
and~\eqref{eq:lambda-intrinsic} which imply
\begin{equation*}
  \theta\le c\lambda
  \qquad\mbox{and}\qquad
  \lambda^p\ge c\bigg(\biint_Q|Du|^{\nu p}\d x\d t\bigg)^{\frac{1}{\nu}}.
\end{equation*}
Then we apply Young's inequality with the power $\frac{q+1}{(1-\nu)p}$
and its conjugate, which are greater than one since $\nu \geq \frac{p-1}{p} $. This concludes the claim for the first term.
  
  For the second term we use Lemma~\ref{lem:gluing} and deduce
  \begin{align*}
\theta^{p\frac{q-1}{q+1}} \biint_{Q}\frac{\big|\big[(\u^q)_{\widehat B}\big]^\frac{1}{q}(t) - \hat a\big|^p}{\rho^{p}}\,\d
    x\d t &\leq c \theta^{p\frac{q-1}{q}} \lambda^\frac{p(2-p)}{q} \left( \biint_{Q} |Du|^{p-1} + |F|^{p-1} \,\d
    x\d t \right)^\frac{p}{q} \\
    &\le c \lambda^{p\frac{q+1-p}{q}} \left( \biint_{Q} |Du|^{p-1} + |F|^{p-1} \,\d
    x\d t \right)^\frac{p}{q} \\
    &\le c \left(\biint_{Q} |Du|^{p-1} + |F|^{p-1} \,\d
    x\d t \right)^\frac{p}{p-1},
  \end{align*}
  since assumptions \eqref{eq:theta-degenerate-p}
  and~\eqref{eq:lambda-intrinsic} imply $\theta\le
  c\lambda$ and $p > q+1$, which concludes the proof in this case. 

\emph{Case 2: The $\theta$-intrinsic case~\eqref{eq:theta-intrinsic-p}.}
By using triangle inequality and Lemma~\ref{lem:uavetoa} with $\alpha=1$, we write
\begin{align*}
\theta^{p\frac{q-1}{q+1}}
  \biint_{Q}\frac{| u - a |^p}{ \rho^{p}} \, \d x \d t
  &\leq c\theta^{p\frac{q-1}{q+1}}
        \biint_{Q}\frac{\big| u -
        \big[(\u^q)_{\widehat B}\big]^\frac{1}{q}(t)\big|^p}{ \rho^{p} } \, \d x \d t \\
  &\phantom{+}
    +c\frac{\theta^{2p\frac{q-1}{q+1}}}{\theta^{p\frac{q-1}{q+1}}}
    \bint_{\Lambda}\frac{\big|\big[(\u^q)_{\widehat B}\big]^\frac{1}{q}(t)
    - \big[(\u^q)_{\widehat Q}\big]^\frac{1}{q}\big|^p}{ \rho^{p} } \,  \d t \\
&=: \mathrm{I} + \mathrm{II}.
\end{align*}
The $\theta$-superintrinsic scaling~\eqref{eq:theta-intrinsic-p}$_2$ implies
\begin{align*}
  \theta^2
  \le
  c\bigg(\frac{|a|}{\rho}\bigg)^{q+1}
  +
  c\bigg(\biint_{Q} \frac{|u-a|^{p}}{\rho^{p}} \, \d x \d t\bigg)^{\frac{q+1}{p}}.
\end{align*}
We use this to estimate the term $\mathrm{I}$ and apply
Lemma~\ref{lem:uavetoa} with $\alpha=\frac1q$ and $p$. Note
that the application is possible since $p> q+1>q$. This procedure leads to the estimate
\begin{align*}
\mathrm I &\leq c\left( \frac{|a|}{\rho} \right)^{p\frac{q-1}{2}} \biint_{Q} \frac{|u - (u)_B (t)|^p}{\rho^p} \, \d x  \d t \\
&\phantom{+} + c\left(\biint_{Q} \frac{\big|u-\big[(\u^q)_{\widehat B}\big]^\frac{1}{q}(t)\big|^{p}}{\rho^{p}} \, \d x \d t\right)^{\frac{q-1}{2}} \biint_{Q} \frac{|u-(u)_B(t)|^{p}}{\rho^{p}} \, \d x \, \d t \\
&\phantom{+} + c\left(\biint_{Q} \frac{\big|\big[(\u^q)_{\widehat B}\big]^\frac{1}{q}(t)-\hat a\big|^{p}}{\rho^{p}} \, \d x \d t\right)^{\frac{q-1}{2}} \biint_{Q} \frac{|u-(u)_B(t)|^{p}}{\rho^{p}} \, \d x \, \d t  \\
&=: \mathrm{I}_1  + \mathrm{I}_2 + \mathrm{I}_3,
\end{align*}
where we abbreviated $\hat a = [(\u^q)_{\widehat
  Q}]^\frac{1}{q}$. By using Lemma~\ref{lem:GN} with $(p,q,r,\vartheta)$ replaced by
$(p,\nu p,2,\nu)$, which is possible since $\nu \ge \frac{n}{n+2}$, we have
\begin{align*}
  \mathrm{I}_1
  \leq
  c\left(\frac{|a|}{\rho}\right)^{p\frac{q-1}{2}}
  \theta^{-p \frac{q-1}{q+1}}&\biint_{Q} \Bigg[|D u|^{\nu p} + \frac{ \babs{ u - (u)_{B}(t)}^{\nu p}}{\left( \theta^{\frac{1-q}{1+q}} \rho \right)^{\nu p}}\Bigg] \, \d x \d t \\
&\phantom{+}\cdot\left(\sup_{t \in\Lambda}\bint_{B}\frac{\babs{u-(u)_{B}(t)}^{2}}{\left(\theta^{\frac{1-q}{1+q}}\rho\right)^{2}}\,\d x \right)^\frac{(1-\nu)p}{2}. 
\end{align*}
This is exactly the same estimate as~\eqref{first-bound-I1} in the
proof of Lemma~\ref{lem:sobo-poincare-p}. Therefore, we can repeat the
arguments leading to~\eqref{bound-I-1} and obtain
\begin{align*}
  \mathrm{I}_1
  \le
  \eps \sup_{t \in \Lambda} \bint_{B} \frac{ \babs{ \u^\frac{q+1}{2} - \a^\frac{q+1}{2}}^{2}}{ \lambda^{2-p} \rho^{q+1}} \, \d x +  c\eps^{-\beta} \left( \biint_{Q} |D u|^{\nu p} \, \d x \d t \right)^\frac{1}{\nu}.
\end{align*}

Next, we estimate the term $\mathrm{I}_2$. Observe that Lemma~\ref{lem:uavetoa} implies
\begin{equation*}
\mathrm{I}_2 \leq c \left( \biint_Q \frac{|u-(u)_B(t)|^{p}}{\rho^{p}}  \, \d x\d t \right)^\frac{q+1}{2}.
\end{equation*}
Furthermore, by applying Lemma~\ref{lem:p-slicewise-estimate} and~\eqref{eq:lambda-intrinsic}$_1$ we have
\begin{align*}
\mathrm{I}_2 &\leq c \lambda^\frac{p(2-p)(1-\nu)}{2+\nu(q-1)} \bigg(\biint_{Q}|Du|^{\nu p}\dx\dt\bigg)^{\frac{q+1}{2+\nu (q-1)}}\left( \sup_{t\in \Lambda}\bint_{B} \frac{|u-a|^{q+1}}{\lambda^{2-p}\rho^{q+1}} \, \d x \right)^{\frac{p(1-\nu)}{2+\nu (q-1)}}  \\
&\leq c \bigg( \left[\biint_{Q}|Du|^{\nu p}\dx\dt \right]^\frac{1}{\nu}\bigg)^{\frac{(2-p)(1-\nu) + \nu(q+1)}{2+\nu (q-1)}} \left( \sup_{t\in \Lambda}\bint_{B} \frac{|u-a|^{q+1}}{\lambda^{2-p}\rho^{q+1}} \, \d x \right)^{\frac{p(1-\nu)}{2+\nu (q-1)}}.
\end{align*}
Since $\nu \geq \frac{p-1}{p}$ the exponents outside the round brackets are less than one, and furthermore they add up to one. Thus, we may use Young's inequality which completes the treatment of the term $\mathrm{I}_2$.

Then we consider the term $\mathrm{I}_3$. By using Lemma~\ref{lem:gluing} for the first term and Poincar\'e inequality for the second we obtain
\begin{align*}
\mathrm{I}_3 \leq c \theta^{-p\frac{q-1}{2q}} \lambda^{p\frac{2q+(2-p)(q-1)}{2q}} \bigg(\biint_Q|Du|^{p-1}+|F|^{p-1}\dx\dt\bigg)^{\frac{p(q-1)}{2q}}
\end{align*}
This corresponds to estimate~\eqref{first-bound-I-22} for the term
$\mathrm{I}_{2,2}$ in the proof of
Lemma~\ref{lem:sobo-poincare-p}. Therefore, arguing as after
estimate~\eqref{first-bound-I-22}, we deduce 
\begin{align*}
\mathrm{I}_{3} \leq \eps \lambda^p + c\eps^{-\beta} \bigg(\biint_Q|Du|^{p-1}+|F|^{p-1}\dx\dt\bigg)^{\frac{p}{p-1}}.
\end{align*}

By the $\theta$-superintrinsic
scaling~\eqref{eq:theta-intrinsic-p}$_2$ we have
\begin{align*}
  \theta^2
  &\le
  c\left( \frac{|\hat a|}{\rho} \right)^{q+1}
  +
  c \left( \biint_{Q} \frac{\big|u - \big[(\u^q)_{\widehat B}\big]^\frac{1}{q}(t)\big|^{p}}{\rho^{p}} \, \d x \d t \right)^\frac{q+1}{p}\\ 
  &\qquad+
  c \left(\mint_{\Lambda} \frac{\big|\big[(\u^q)_{\widehat B}\big]^\frac{1}{q}(t) - \hat a\big|^{p}}{\rho^{p}} \, \d t\right)^\frac{q+1}{p},
\end{align*}
where $\hat a = [(\u^q)_{\widehat
  Q}]^\frac{1}{q}$. Using this for the estimate of $\mathrm{II}$, we
obtain 
\begin{align*}
\mathrm{II} 
&\leq c \theta^{-p \frac{q-1}{q+1}}\left( \frac{|\hat a|}{\rho} \right)^{(q-1)p}\bint_{\Lambda}\frac{\big|\big[(\u^q)_{\widehat B}\big]^\frac{1}{q}(t) - \hat a \big|^p}{ \rho^{p} } \,  \d t \\
&\phantom{+} +c\theta^{-p \frac{q-1}{q+1}}\left( \biint_{Q} \frac{\big|u - \big[(\u^q)_{\widehat B}\big]^\frac{1}{q}(t)\big|^{p}}{\rho^{p}} \, \d x \d t \right)^{q-1}\bint_{\Lambda}\frac{\big|\big[(\u^q)_{\widehat B}\big]^\frac{1}{q}(t) - \hat a \big|^p}{ \rho^{p} } \,  \d t \\
&\phantom{+} +c\theta^{-p \frac{q-1}{q+1}}\left( \bint_{\Lambda} \frac{\big[|(\u^q)_{\widehat B}\big]^\frac{1}{q}(t) - \hat a\big|^{p}}{\rho^{p}} \, \d t \right)^{q-1}\bint_{\Lambda}\frac{\big|\big[(\u^q)_{\widehat B}\big]^\frac{1}{q}(t) - \hat a \big|^p}{ \rho^{p} } \,  \d t \\
&=: \mathrm{II}_1 + \mathrm{II}_2 + \mathrm{II}_3.
\end{align*}

For the first term, we use Lemma~\ref{lem:Acerbi-Fusco}, which implies
\begin{align*}
  \mathrm{II}_1 \leq c \theta^{-p \frac{q-1}{q+1}}\bint_{\Lambda}\frac{| (\u^q)_{\widehat B}(t) - \hat \a^q |^p}{ \rho^{pq} } \,  \d t,
\end{align*}
while the third term is estimated with the help of Lemma~\ref{lem:a-b}
and H\"older's inequality, which gives
\begin{align*}
  \mathrm{II}_3
  &\leq c \theta^{-p \frac{q-1}{q+1}} \left( \bint_{\Lambda}\frac{|
    (\u^q)_{\widehat B}(t) - \hat \a^q |^\frac{p}{q}}{ \rho^{p} } \,  \d t
    \right)^q\\
  &\leq c \theta^{-p \frac{q-1}{q+1}}\bint_{\Lambda}\frac{| (\u^q)_{\widehat B}(t) - \hat \a^q |^p}{ \rho^{pq} } \,  \d t.
\end{align*}
Therefore, both terms can be estimated as in~\eqref{first-bound-II-1},
with the result
\begin{equation*}
  \mathrm{II}_1+\mathrm{II}_3
  \le
  c \left( \biint_{Q} |Du|^{\nu p} \, \d x \d t \right)^\frac{1}{\nu} + c \biint_{Q} |F|^p \, \d x \d t. 
\end{equation*}
For the term $\mathrm{II}_2$, we
estimate the first part using the $\theta$-subintrinsic
scaling~\eqref{eq:theta-intrinsic-p}$_1$ and for the last integral we
apply Lemma~\ref{lem:a-b} with $\alpha=q$. The
resulting integrals are then estimated by Lemma~\ref{lem:p-slicewise-estimate} 
and Lemma~\ref{lem:gluing}, respectively.  This yields
 \begin{align*}
   \mathrm{II}_2
   &\le
   c\theta^{-\frac{p(q-1)}{q(q+1)}}\left( \biint_{Q} \frac{\big|u -(u)_B(t)\big|^{p}}{\rho^{p}} \,
      \d x \d t \right)^{\frac{(q-1)(q+1)}{2q}}\bint_{\Lambda}\frac{|(\u^q)_{\widehat B}(t) - \hat \a^q |^{\frac{p}{q}}}{ \rho^{p} } \,\d t\\
  &\leq c \theta^{-\frac{p(q-1)}{q(q+1)}}  \lambda^\frac{p(q-1)}{q} \left( \lambda^{2-p} \theta^\frac{q-1}{q+1} \biint_Q |Du|^{p-1} + |F|^{p-1} \, \d x \d t \right)^\frac{p}{q} \\
  &= c \lambda^\frac{p(q+1-p)}{q} \left( \biint_Q |Du|^{p-1} + |F|^{p-1} \, \d x \d t \right)^\frac{p}{q} \\
  &\le \eps \lambda^p + c\eps^{-\beta} \left( \biint_Q |Du|^{p-1} + |F|^{p-1} \, \d x \d t  \right)^\frac{p}{p-1},
 \end{align*}
 where we also used Young's inequality with exponents $\frac{q}{q+1-p}$ and $\frac{q}{p-1}$ on the last line. Thus the claim follows.
\end{proof}

Finally, we state the Sobolev-Poincar\'e inequality for the second
term on the right-hand side of~\eqref{energy-estimate}. It turns out
that its proof can be reduced to the preceding
Lemma~\ref{lem:sobo-poincare-p-2}.

\begin{lemma} \label{lem:sobo-poincare-q+1-2}
Suppose that $p > q+1 > 2$ 
  and that $u$ is a weak solution to~\eqref{eq:dne}, where 
  assumption~\eqref{assumption:A} holds true.
  Moreover, we consider a cylinder $Q_{2\rho}^{(\lambda,\theta)}(z_o)
\Subset \Omega_T$ and assume that~\eqref{eq:lambda-intrinsic} together with
either~\eqref{eq:theta-intrinsic-p} or~\eqref{eq:theta-degenerate-p}
are satisfied. Then the following Sobolev-Poincar\'e inequality holds:
\begin{align*}
  \lambda^{p-2} &\biint_{Q_\rho^{(\lambda,\theta)}(z_o)}
          \frac{\babs{ \u^\frac{q+1}{2}-\a^\frac{q+1}{2} }^2}{ \rho^{q+1} } \, \d x \d t \\
&\leq \eps \left( \sup_{t \in \Lambda_\rho^{(\lambda)}(t_o)} \lambda^{p-2} \bint_{B_\rho^{(\theta)}(x_o)} \frac{\babs{ \u^\frac{q+1}{2}(t)-\a^\frac{q+1}{2} }^2}{ \rho^{q+1} } \, \d x + \biint_{Q_\rho^{(\lambda,\theta)}(z_o)} |Du|^p \, \d x \d t \right) \\
&\phantom{+} + c \eps^{-\beta} \left[ \left( \biint_{Q_\rho^{(\lambda,\theta)}(z_o)} |Du|^{\nu p} \, \d x \d t  \right)^\frac{1}{\nu} + \biint_{Q_\rho^{(\lambda,\theta)}(z_o)} |F|^p \, \d x \d t  \right],
\end{align*}
where $\max \left\{ \frac{p-1}{p},\frac{n}{n+2} \right\} \leq \nu \leq 1$ and $a = (u)_{z_o;\rho}^{(\theta,\lambda)}$. The preceding
estimate holds for an arbitrary $\eps\in(0,1)$ with a constant
$c = c(n,p,q,C_1,C_\theta,C_\lambda)>0$ and $\beta = \beta (n,p,q)>0$.
\end{lemma}

\begin{proof}

Observe that $p > q+1 > 2$. Applying Lemma~\ref{lem:Acerbi-Fusco} and H\"older's inequality with exponents $\frac{q+1}{q-1}$ and $\frac{q+1}{2}$, we estimate
\begin{align*}
\lambda^{p-2} &\biint_{Q}
          \frac{\babs{ \u^\frac{q+1}{2}-\a^\frac{q+1}{2} }^2}{ \rho^{q+1} } \, \d x \d t \\
          &\leq c\lambda^{p-2} \left( \biint_{Q} \frac{|u|^{q+1}}{\rho^{q+1}} \, \d x\d t\right)^\frac{q-1}{q+1}
          \left( \biint_Q  \frac{|u-a|^{q+1}}{\rho^{q+1}} \, \d x \d t \right)^\frac{2}{q+1}.
\end{align*}
By using H\"older's inequality, $\theta$-subintrinsic scaling~\eqref{eq:theta-intrinsic-p}$_1$ for the first term and using Young's inequality with exponents $\frac{p}{p-2}$ and $\frac{p}{2}$ we further obtain
\begin{align*}
\lambda^{p-2} \biint_{Q}
          \frac{\babs{ \u^\frac{q+1}{2}-\a^\frac{q+1}{2} }^2}{ \rho^{q+1} } \, \d x \d t &\leq c\lambda^{p-2} \theta^{2 \frac{q-1}{q+1}}
          \left( \biint_Q  \frac{|u-a|^{p}}{\rho^{p}} \, \d x \d t  \right)^\frac{2}{p} \\
          &\le \eps \lambda^p + c\eps^{-\beta} \theta^{p \frac{q-1}{q+1}}
          \biint_Q  \frac{|u-a|^{p}}{\rho^{p}} \, \d x \d t.
\end{align*}
The claim follows by using Lemma~\ref{lem:sobo-poincare-p-2} for the latter term.
\end{proof}

\section{Reverse H\"older inequality}
\label{sec:reverse-Holder}

In the next lemma, we combine the energy
estimate~\eqref{energy-estimate} with the Sobolev-Poincar\'e
inequalities from the preceding sections to prove a reverse H\"older
inequality that will be a crucial tool for the proof of the higher integrability.

\begin{lemma} \label{lem:reverse-holder}
Let $q > 1$, $p > \frac{n(q+1)}{n+q+1}$ and $u$ be a weak solution
to~\eqref{eq:dne} in the sense of Definition~\ref{def:weak-sol} and
let $Q_{2\rho}^{(\lambda,\theta)}(z_o) \Subset \Omega_T$ be a cylinder 
for some $\rho > 0$, $\lambda >0$ and $\theta > 0$.
If~\eqref{eq:lambda-intrinsic} together with~\eqref{eq:theta-intrinsic} or~\eqref{eq:theta-degenerate-1}
are satisfied, then the following reverse H\"older inequality holds true
\begin{align*}
\biint_{Q_{\rho}^{(\lambda,\theta)}(z_o)} |Du|^p \, \d x \d t \leq c \left( \biint_{Q_{2\rho}^{(\lambda,\theta)}(z_o)} |Du|^{\nu p} \, \d x \d t \right)^\frac{1}{\nu}	+ c \biint_{Q_{2\rho}^{(\lambda,\theta)}(z_o)} |F|^p \, \d x \d t,
\end{align*}
for $\max \left\{\frac{p-1}{p}, \frac{n}{n+2},\frac{n}{n+2}
  \left(1+\frac{2}{p}- \frac{2}{q} \right), \frac{n(q+1)}{p(n+q+1)}
\right\} \leq \nu \leq 1$ and a constant $c > 0$ depending on $n,p,q,C_o, C_1, C_\lambda, C_\theta$. 
\end{lemma}

\begin{proof}

We omit the center point $z_o$ from the notation for simplicity. Let $\rho \leq r < s \leq 2\rho$ and denote $a_\sigma = (u)_\sigma^{(\lambda,\theta)}$ for $\sigma \in \{r,s\}$. Lemma~\ref{lem:caccioppoli} implies 
\begin{align*}
\sup_{t\in \Lambda_r^{(\lambda)}}&\, \bint_{B_r^{(\theta)}} \frac{\babs{ \u^\frac{q+1}{2}(t)-\a_r^\frac{q+1}{2} }^2}{\lambda^{2-p} r^{q+1}} \d x + \biint_{Q_r^{(\lambda,\theta)}} |Du|^p\, \d x\d t \\
& \leq c \biint_{Q_s^{(\lambda,\theta)}} \left[ \theta^\frac{p(q-1)}{q+1} \frac{|u-a_r |^p}{(s-r)^p}+ \frac{\babs{ \u^\frac{q+1}{2}-\a_r^\frac{q+1}{2} }^2}{ \lambda^{2-p} (s^{q+1}-r^{q+1} )} + |F|^p  \right] \d x\d t  \\
&\leq c \mathcal R_{r,s}^p \biint_{Q_s^{(\lambda,\theta)}}  \theta^\frac{p(q-1)}{q+1} \frac{|u-a_s |^p}{s^p} \, \d x \d t + c \mathcal R_{r,s}^{q+1} \biint_{Q_s^{(\lambda,\theta)}} \frac{\babs{ \u^\frac{q+1}{2}-\a_s^\frac{q+1}{2} }^2}{ \lambda^{2-p} s^{q+1} } \, \d x \d t \\
&\qquad + \biint_{Q_s^{(\lambda,\theta)}} |F|^p \, \d x \d t \\
&=: \mathrm{I} + \mathrm{II} + \mathrm{III},
\end{align*}
by using also Lemma~\ref{lem:uavetoa} and denoting $\mathcal R_{r,s} = \frac{s}{s-r}$.
We apply Lemma~\ref{lem:sobo-poincare-p} for I and Lemma~\ref{lem:sobo-poincare-q+1} for II if $q+1 \geq p$, and Lemmas~\ref{lem:sobo-poincare-p-2} and~\ref{lem:sobo-poincare-q+1-2} respectively if $p > q+1$, which yields
\begin{align*}
\sup_{t\in \Lambda_r^{(\lambda)}}&\, \bint_{B_r^{(\theta)}} \frac{\babs{ \u^\frac{q+1}{2}(t)-\a_r^\frac{q+1}{2} }^2}{\lambda^{2-p} r^{q+1}} \d x + \biint_{Q_r^{(\lambda,\theta)}} |Du|^p\, \d x\d t \\
&\leq \eps c \mathcal R_{r,s}^{p^\sharp} \left( \sup_{t \in \Lambda_s^{(\lambda)}}  \bint_{B_s^{(\theta)}} \frac{\babs{ \u^\frac{q+1}{2}(t)-\a_s^\frac{q+1}{2} }^2}{\lambda^{2-p} s^{q+1} } \, \d x + \biint_{Q_s^{(\lambda,\theta)}} |Du|^p \, \d x \d t \right) \\
&\phantom{+} + \eps^{-\beta} c \mathcal R_{r,s}^{p^\sharp}\left[ \left( \biint_{Q_{2\rho}^{(\lambda,\theta)}} |Du|^{\nu p} \, \d x \d t  \right)^\frac{1}{\nu} + \biint_{Q_{2\rho}^{(\lambda,\theta)}} |F|^p \, \d x \d t  \right],
\end{align*}
for every $\eps\in(0,1)$. 
We fix $\eps = \frac{1}{2 c \mathcal R_{r,s}^{p^\sharp}}$, and use Lemma~\ref{lem:iteration} to conclude the result.
\end{proof}

We end this section with a technical lemma that will be needed to
prove the $\theta$-singular scaling~\eqref{eq:theta-degenerate-1} in
the cases in which the $\theta$-intrinsic
scaling~\eqref{eq:theta-intrinsic} is not available, see
Section~\ref{sec:reverse-Holder-final-section}.

\begin{lemma}\label{lem:theta-absorb}
Let $q > 1$, $p > \frac{n(q+1)}{n+q+1}$ and $u$ be a weak solution
to~\eqref{eq:dne} in the sense of Definition~\ref{def:weak-sol} and
let $Q_{2\rho}^{(\lambda,\theta)}(z_o) \Subset \Omega_T$ be a cylinder
for some $\rho>0$, $\lambda >0$ and $\theta > 0$.
If~\eqref{eq:lambda-intrinsic}$_1$ and~\eqref{eq:theta-intrinsic} with
$C_\theta = 1$  are satisfied, we have
$$
\theta^\frac{2}{q+1} \leq c \lambda^\frac{2}{q+1} + \frac{3}{4} \left( \biint_{Q_{\rho/2}^{(\lambda,\theta)}(z_o)} \frac{|u|^{p^\sharp}}{(\rho/2)^{p^\sharp}} \, \d x \d t \right)^\frac{1}{p^\sharp}
$$
for $c = c(n,p,q,C_o,C_1,C_\lambda) > 0$.
\end{lemma}

\begin{proof}
We apply first~\eqref{eq:theta-intrinsic}$_2$ with $C_\theta = 1$,
then the triangle inequality and Lemma~\ref{lem:uavetoa}, and finally,
the triangle inequality again. In this way, we get
\begin{align*}
  \theta^\frac{2}{q+1}
  &\leq \left( \biint_{Q_\rho^{(\lambda,\theta)}}
    \frac{|u|^{p^\sharp}}{\rho^{p^\sharp}} \, \d x \d t
    \right)^\frac{1}{p^\sharp}\\
  &\leq c(n,p,q) \left( \biint_{Q_\rho^{(\lambda,\theta)}} \frac{\big|u -
    (\u^q)_{\widehat Q}^\frac{1}{q}\big|^{p^\sharp}}{\rho^{p^\sharp}} \,
    \d x \d t \right)^\frac{1}{p^\sharp}
    + \frac{\Big|(\u^q)_{Q_{\rho/2}^{(\lambda,\theta)}}\Big|^\frac{1}{q}}{\rho} \\
  &\leq c(n,p,q) \left( \biint_{Q_\rho^{(\lambda,\theta)}} \frac{\big|u - (\u^q)_{\widehat B}^\frac{1}{q}(t)\big|^{p^\sharp}}{\rho^{p^\sharp}} \, \d x \d t \right)^\frac{1}{p^\sharp} \\
&\qquad+ c(n,p,q)\left( \biint_{Q_\rho^{(\lambda,\theta)}} \frac{\big|(\u^q)^\frac{1}{q}_{\widehat B}(t) - (\u^q)_{\widehat Q}^\frac{1}{q}\big|^{p^\sharp}}{\rho^{p^\sharp}} \, \d x \d t \right)^\frac{1}{p^\sharp} \\
&\qquad+ \frac{\Big|(\u^q)_{Q_{\rho/2}^{(\lambda,\theta)}}\Big|^\frac{1}{q}}{\rho} \\
&=: \mathrm{I} + \mathrm{II} + \mathrm{III}.
\end{align*}
Here we used the abbreviations $\widehat B=B_{\hat\rho}^{(\theta)}$
and $\widehat Q:=\widehat B\times\Lambda_\rho^{(\lambda)}$, with the
radius $\hat\rho\in[\frac\rho2,\rho]$ provided by
Lemma~\ref{lem:gluing}. 
Observe that by H\"older's inequality
$$
\mathrm{III} \leq \frac{1}{2} \left( \biint_{Q_{\rho/2}^{(\lambda,\theta)}} \frac{|u|^{p^\sharp}}{(\rho/2)^{p^\sharp}} \, \d x \d t \right)^\frac{1}{p^\sharp}.
$$
By Lemmas~\ref{lem:a-b},~\ref{lem:uavetoa} and~\ref{lem:gluing} we obtain
\begin{align*}
\mathrm{II} &\leq c(n,p,q) \rho^{-1} \sup_{t,\tau \in \Lambda_{\rho}^{(\lambda)}} |(\u^q)_{\widehat B}(t) - (\u^q)_{\widehat B}(\tau)|^\frac{1}{q} \\
&\leq c(n,p,q,C_1) \lambda^\frac{2-p}{q} \theta^\frac{q-1}{q(q+1)} \left( \biint_{Q_{\rho}^{(\lambda,\theta)}} |Du|^{p-1} + |F|^{p-1} \, \d x \d t \right)^\frac{1}{q} \\
&\leq c(n,p,q,C_1,C_\lambda) \lambda^\frac{1}{q} \theta^\frac{q-1}{q(q+1)} \leq \eps \theta^\frac{2}{q+1} + c_\eps \lambda^\frac{2}{q+1},
\end{align*}
in which $c_\eps$ depends on $\eps, n,p,q,C_1$ and $C_\lambda$. On the
last line we also used~\eqref{eq:lambda-intrinsic}$_1$ and Young's
inequality with exponents $\frac{2q}{q+1}$ and $\frac{2q}{q-1}$.

For the estimate of $\mathrm{I}$, we consider the case $p > q+1$ first. In
this case, Lemmas~\ref{lem:uavetoa} and~\ref{lem:p-slicewise-estimate}
imply 
\begin{align*}
\mathrm{I} \leq c \left( \biint_{Q_\rho^{(\lambda,\theta)}} \frac{\big|u - (u)_{\rho}^{(\theta)}(t)\big|^{p}}{\rho^{p}} \, \d x \d t \right)^\frac{1}{p} \leq c \lambda^\frac{2}{q+1}
\end{align*}
for $c = c(n,p,q,C_\lambda)$.
Then let us consider the case $q+1 \geq p$. By using Lemma~\ref{lem:q+1-slicewise-estimate} with $a = 0$ we have
\begin{equation} \label{eq:I-est-absorb}
\mathrm{I} \leq c \lambda^{ \frac{2}{q+1} \cdot\frac{2(q+1)+ p(p-2)}{2(q+1) +p(q-1)}} \left( \sup_{t\in \Lambda_\rho^{(\lambda)}} \bint_{B_\rho^{(\theta)}} \frac{|u|^{q+1}}{\lambda^{2-p}\rho^{q+1}} \, \d x \right)^{\frac{2}{q+1} \cdot\frac{q+1-p}{2(q+1) + p(q-1)}}
\end{equation}
for $c = c(n,p,q,C_\lambda)$. By using the energy estimate from Lemma~\ref{lem:caccioppoli} with $a= 0$ we obtain
\begin{align*}
\sup_{t\in \Lambda_\rho^{(\lambda)}} \bint_{B_\rho^{(\theta)}} \frac{|u|^{q+1}}{\lambda^{2-p}\rho^{q+1}} \, \d x &\leq c \biint_{Q_{2\rho}^{(\lambda,\theta)}} \theta^\frac{p(q-1)}{q+1} \frac{|u|^p}{\rho^p} + \lambda^{p-2} \frac{|u|^{q+1}}{\rho^{q+1}} + |F|^p \, \d x \d t \\
&\leq c \left[ \theta^p + \lambda^{p-2} \theta^2 + \lambda^p \right]
\end{align*}
for $c = c(p,q,C_o,C_1,C_\lambda)$, where we also
used~\eqref{eq:theta-intrinsic}$_1$
and~\eqref{eq:lambda-intrinsic}$_1$. By plugging this
into~\eqref{eq:I-est-absorb}, observing that $\frac{2(q+1)+
  p(p-2)}{2(q+1) +p(q-1)} + p \frac{q+1-p}{2(q+1) + p(q-1)} = 1$, we
use Young's inequality to the first two terms including $\theta$ to
conclude
\begin{equation*}
  \mathrm{I}
  \le
  \eps \theta^\frac{2}{q+1} + c_\eps \lambda^\frac{2}{q+1},  
\end{equation*}
in which $c_\eps$ depends on $\eps, n,p,q,C_o,C_1$ and
$C_\lambda$. Collecting the estimates, we obtain in any case
\begin{equation*}
  \theta^\frac{2}{q+1}
  \leq
  2\eps \theta^\frac{2}{q+1} + c_\eps \lambda^\frac{2}{q+1}
  +
  \frac{1}{2} \left( \biint_{Q_{\rho/2}^{(\lambda,\theta)}} \frac{|u|^{p^\sharp}}{(\rho/2)^{p^\sharp}} \, \d x \d t \right)^\frac{1}{p^\sharp}.
\end{equation*}
By choosing $\eps = \frac{1}{6}$ the claim follows.

\end{proof}

\section{Proof of the higher integrability}
\label{sec:proof-main-thm}

This section is devoted to the proof of our main result,
Theorem~\ref{thm:main}.
Fix $Q_{4R}$ with $R > 0$ such that $Q_{8R} \Subset \Omega_T$ and
\begin{equation}\label{choice-lambda_o}
\lambda_o \geq 1 + \left( \biint_{Q_{4R}}
  \frac{|u|^{p^\sharp}}{(4R)^{p^\sharp}} \, \d x \d t
\right)^\frac{d}{p},
\end{equation}
where the parameter $d\ge1$ is defined in~\eqref{def-d}. Note that we
can rewrite it as 
\begin{equation*}
  d=\frac{p(q+1)}{(q+1)^2+(p^\sharp+n)(p-p^\sharp)}.
\end{equation*}

Fix $\lambda \geq \lambda_o$ and 
\begin{equation}\label{choice-R_o}
  R_o=\min\{\lambda^\frac{p-2}{q+1},\lambda^\frac{q-1}{q+1}\}R=\lambda^{\frac{p+q-1-p^\sharp}{q+1}}R.
\end{equation}
Note that $R_o$ might be larger than $R$ for certain values of
parameters, but by definition of $Q_{2\rho}^{(\lambda,\theta)}(z_o)$, 
we still have the inclusion
\begin{equation*}
  Q_{2\rho}^{(\lambda,\theta)}(z_o)
  \subset
 Q_{2R}(z_o)\subset Q_{4R} 
\end{equation*}
for every $z_o \in Q_{2R}$, $\theta
\geq \lambda$ and $\rho \leq R_o$.

The crucial step of the proof is to construct a suitable family of
parabolic cylinders, which satisfy a Vitali type covering property and
for which~\eqref{eq:lambda-intrinsic} and either
~\eqref{eq:theta-intrinsic}
or~\eqref{eq:theta-degenerate-1} hold true, so that the reverse
H\"older inequality from Lemma~\ref{lem:reverse-holder} is applicable.

\subsection{Construction of a non-uniform system of cylinders}
\label{subsec:cylinder-construction}
For fixed $z_o \in Q_{2R}$, $\lambda \geq \lambda_o$ and $\rho \in (0,R_o]$ we define
$$
\tilde \theta^{(\lambda)}_{z_o; \rho} := \inf \left\{ \theta \in [\lambda,\infty): \frac{1}{|Q_\rho|} \iint_{Q_\rho^{(\lambda,\theta)}(z_o)} \frac{|u|^{p^\sharp}}{\rho^{p^\sharp}}\, \d x\d t \leq \lambda^{2-p} \theta^\frac{2p^\sharp + n (1-q)}{1+q}  \right\}.
$$
Observe that the integral above converges to zero when $\theta \to
\infty$, while the right hand side blows up with speed
$\theta^\frac{2p^\sharp + n (1-q)}{1+q}$ provided that $q <
\frac{n+2}{n-2}$ if $p \leq q+1$, and $p > \frac{n}{2}(q-1)$ if $p >
q+1$. Thus, there exists a unique $\tilde \theta^{(\lambda)}_{z_o;
  \rho}$ for fixed $z_o,\rho$ and $\lambda$ satisfying the above conditions. In case $\lambda$ and $z_o$ are clear from the context, we omit them from the notation.

By definition, one of the following two alternatives occur; either 
\begin{equation*}
\tilde \theta_\rho = \lambda \quad \text{and}\quad \biint_{Q_\rho^{(\lambda,\tilde \theta_{\rho})}(z_o)}\frac{|u|^{p^\sharp}}{\rho^{p^\sharp}} \, \d x \d t \leq \tilde \theta_\rho^\frac{2p^\sharp}{q+1} = \lambda^\frac{2p^\sharp}{q+1},
\end{equation*}
or
\begin{equation} \label{eq:theta-tilde-intrinsic}
\tilde \theta_\rho > \lambda \quad \text{and}\quad \biint_{Q_\rho^{(\lambda,\tilde \theta_\rho)}(z_o)}\frac{|u|^{p^\sharp}}{\rho^{p^\sharp}} \, \d x \d t = \tilde \theta_\rho^\frac{2p^\sharp}{q+1}.
\end{equation}
Note that if $\tilde \theta_{R_o} > \lambda$, it follows from
\eqref{choice-lambda_o} that
\begin{align} \label{eq:theta-R0-lambda}
\tilde \theta_{R_o}^\frac{2p^\sharp + n(1-q)}{q+1} &= \frac{\lambda^{p-2}}{|Q_{R_o}|} \iint_{Q_{R_o}^{(\lambda,\tilde \theta_{R_o})}(z_o)} \frac{|u|^{p^\sharp}}{R_o^{p^\sharp}} \, \d x \d t \nonumber \\
&\leq \lambda^{p-2} \left( \frac{R}{R_o} \right)^{n+p^\sharp+q+1} \biint_{Q_{R}(z_o)} \frac{|u|^{p^\sharp}}{R^{p^\sharp}} \, \d x \d t \nonumber \\
&\leq 4^{n+p^\sharp+q+1} \lambda^{p-2-(n+p^\sharp+q+1)\frac{p+q-1-p^\sharp}{q+1}} \lambda_o^\frac{p}{d} \nonumber \\
&\leq 4^{n+p^\sharp+q+1} \lambda^{\frac{2p^\sharp + n(1-q)}{q+1}}.
\end{align}
In the last estimate, we distinguished between the cases $p\ge q+1$
and $\frac{n(q+1)}{n+q+1}<p<q+1$ and used the fact $\lambda\ge\lambda_o$.

The mapping $(0,R_o] \ni \rho \mapsto \tilde \theta_\rho$ is continuous by a similar argument as in~\cite{Boegelein-Duzaar-Scheven:2022} (see also~\cite{Boegelein-Duzaar-Scheven:2018, Boegelein-Duzaar-Korte-Scheven,BDKS:2018}) but it is not non-increasing in general.
Therefore, we define 
$$
\theta_{z_o; \rho}^{(\lambda)} := \max_{r\in [\rho, R_o]} \tilde \theta_{z_o;r}^{(\lambda)},
$$
which is clearly continuous (since $\tilde \theta_\rho$ is) and non-increasing with respect to $\rho$.
Furthermore, let 
\[ \tilde \rho :=
\begin{cases}
R_o, &\text{if } \theta_\rho = \lambda, \\
\inf \{s \in [\rho,R_o] : \theta_s = \tilde \theta_s \}, &\text{if } \theta_\rho > \lambda.
\end{cases}
\]
Observe that $\theta_r = \tilde \theta_{\tilde \rho}$ for every $r \in
[\rho,\tilde \rho]$. The following lemma summarizes some basic
properties of the parameter $\theta_\rho$. 

\begin{lemma}  \label{lem:theta-rho-prop}
Let $\theta_\rho$ be constructed as above. Then we have
\begin{enumerate}[(i)]
\item $
  \displaystyle
  \biint_{Q_s^{(\lambda,\theta_\rho)}} \frac{|u|^{p^\sharp}}{s^{p^\sharp}} \, \d x \d t \leq \theta_\rho^{\frac{2p^\sharp}{q+1}} \quad \text{for every } 0< \rho \leq s \leq R_o,
$

\item $\theta_\rho \leq \left( \frac{s}{\rho} \right)^\frac{(q+1)(n+ p^\sharp +q+ 1)}{2 p^\sharp + n(1-q)} \theta_s \quad \text{for every } 0< \rho \leq s \leq R_o$,

\item $\theta_\rho \leq \left( \frac{4 R_o}{\rho} \right)^\frac{(q+1)(n+ p^\sharp +q+ 1)}{2 p^\sharp + n(1-q)} \lambda \quad \text{for every } 0 < \rho \leq R_o$. 
\end{enumerate}

\end{lemma}

\begin{proof}
(i): Clearly $\tilde \theta_s \leq \theta_s \leq \theta_\rho$, which implies $Q_s^{(\lambda,\theta_\rho)} \subset Q_s^{(\lambda,\tilde \theta_s)}$. Thus
\begin{align*}
\biint_{Q_s^{(\lambda,\theta_\rho)}} \frac{|u|^{p^\sharp}}{s^{p^\sharp}} \, \d x \d t &\leq \left( \frac{\theta_\rho}{\tilde \theta_s} \right)^{n \frac{q-1}{q+1}} \biint_{Q_s^{(\lambda,\tilde \theta_s)}} \frac{|u|^{p^\sharp}}{s^{p^\sharp}} \, \d x \d t \\
&\leq \left( \frac{\theta_\rho}{\tilde \theta_s} \right)^{n \frac{q-1}{q+1}} \tilde \theta_s^\frac{2p^\sharp}{q+1} = \theta_\rho^{n \frac{q-1}{q+1}}\tilde \theta_s^\frac{2p^\sharp + n (1-q)}{q+1} \leq \theta_\rho^\frac{2 p^\sharp}{q+1},
\end{align*}
where we have used the fact $2p^\sharp + n (1-q)>0$ that follows from
the assumption $q<\max\{\frac{n+2}{n-2},\frac{2p}{n}+1\}$. 

(ii): If $\theta_\rho = \lambda$ the claim clearly holds. Suppose that $\lambda < \theta_\rho$ and $s \in [\tilde \rho,R_o]$. We have
\begin{align*}
\theta_\rho^\frac{2 p^\sharp + n(1-q)}{q+1} &= \tilde \theta_{\tilde \rho}^\frac{2 p^\sharp + n(1-q)}{q+1} = \frac{\lambda^{p-2}}{|Q_{\tilde \rho}|} \iint_{Q_{\tilde \rho}^{(\lambda, \theta_{\tilde \rho})}} \frac{|u|^{p^\sharp}}{\tilde \rho^{p^\sharp}} \, \d x \d t \\
&\leq \left( \frac{s}{\tilde \rho} \right)^{n+ p^\sharp + q +1} \frac{\lambda^{p-2}}{|Q_{s}|} \iint_{Q_s^{(\lambda,\theta_s)}} \frac{|u|^{p^\sharp}}{s^{p^\sharp}} \, \d x \d t \\
&\leq \left( \frac{s}{\tilde \rho} \right)^{n+ p^\sharp + q +1} \theta_s^\frac{2 p^\sharp + n(1-q)}{q+1},
\end{align*}
which implies the claim. If $s \in [\rho, \tilde \rho)$, then $\theta_\rho = \theta_s$ and the claim clearly holds. 

(iii): By choosing $s = R_o$ in (ii), and using~\eqref{eq:theta-R0-lambda} (observe that $\theta_{R_o} = \tilde \theta_{R_o}$) we have
$$
\theta_\rho \leq \left( \frac{R_o}{\rho} \right)^\frac{(q+1)(n+ p^\sharp +q+ 1)}{2 p^\sharp + n(1-q)} \theta_{R_o} \leq \left( \frac{4 R_o}{\rho} \right)^\frac{(q+1)(n+ p^\sharp +q+ 1)}{2 p^\sharp + n(1-q)} \lambda,
$$
completing the proof.
\end{proof}

\subsection{Vitali type covering property}
\label{sec:Vitali-covering}

\begin{lemma} \label{lem:vitali}
Let $\lambda \geq \lambda_o$. There exists $\hat c = \hat c (n,p,q) \geq 20$ such that the following holds: Let $\mathcal F$ be any collection of cylinders $Q_{4r}^{(\lambda,\theta_{z;r}^{(\lambda)})} (z)$, where $Q_{r}^{(\lambda,\theta_{z;r}^{(\lambda)})} (z)$ is a cylinder of the form that is constructed in Section~\ref{subsec:cylinder-construction} with radius $r \in (0,\frac{R_o}{\hat c})$. Then, there exists a countable, disjoint subcollection $\mathcal G$ of $\mathcal F$ such that
$$
\bigcup_{Q \in \mathcal F} Q \subset \bigcup_{Q \in \mathcal G} \widehat Q,
$$ 
where $\widehat{Q}$ denotes the $\frac14 \hat c$-times enlarged $Q$, i.e. if $Q = Q_{4r}^{(\lambda,\theta_{z;r}^{(\lambda)})} (z)$ then $\widehat Q = Q_{\hat c r}^{(\lambda,\theta_{z;r}^{(\lambda)})} (z)$.
\end{lemma}

\begin{proof}
As in~\cite{Boegelein-Duzaar-Scheven:2022} (see also~\cite{Boegelein-Duzaar-Scheven:2018, Boegelein-Duzaar-Korte-Scheven,BDKS:2018}) consider
$$
\mathcal F_j := \left\{ Q_{4r}^{(\lambda,\theta_{z;r}^{(\lambda)})} (z) \in \mathcal F : \tfrac{R_o}{2^j \hat c} < r \leq \tfrac{R_o}{2^{j-1} \hat c} \right\}, \quad j \in \N .
$$
Let $\mathcal G_1$ be a maximal disjoint subcollection of $\mathcal F_1$, which is finite by Lemma~\ref{lem:theta-rho-prop} (iii). At stage $k \in \N_{\geq 2}$  let $\mathcal G_k$ be a maximal disjoint collection of cylinders in 
$$
 \left\{ Q\in \mathcal F_k : Q \cap Q^* = \varnothing \text{ for any } Q^* \in \bigcup_{j=1}^{k-1} \mathcal G_j \right\},
$$
and define
$$
\mathcal G = \bigcup_{j=1}^{\infty} \mathcal G_j,
$$
which is countable since $\mathcal G_j$ for every $j \in \N$ is finite.

Our objective to show is that for every $Q \in \mathcal F$ there exists $Q^* \in \mathcal G$ such that $Q \cap Q^* \neq \varnothing$ and $Q \subset \widehat Q^*$. To this end, let $Q = Q_{4r}^{(\lambda,\theta_{z;r}^{(\lambda)})} (z) \in \mathcal F$, which implies that there exists $j \in \N$ such that $Q \in \mathcal F_j$. By maximality of $\mathcal G_j$ there exists $Q^* = Q_{4r_*}^{(\lambda,\theta_{z_*;r_*}^{(\lambda)})} (z_*) \in \bigcup_{i=1}^j \mathcal G_i$ such that $Q \cap Q^* \neq \varnothing$. By definitions of $\mathcal F_j$ and $\mathcal G_j$ it follows that $r < 2 r_*$. This immediately implies
\begin{equation} \label{eq:time-inclusion}
\Lambda_{4r}^{(\lambda)}(t) \subset \Lambda_{12 r_*}^{(\lambda)}(t_*).
\end{equation}
Let $\tilde r_* \in [r_*,R_o]$ be defined as in the earlier section. It follows that
\begin{equation}\label{eq:reverse-time-inclusion}
\Lambda_{4 \tilde r_*}^{(\lambda)}(t_*) \subset
\Lambda^{(\lambda)}_{10 \tilde r_*}(t).
\end{equation}
Next we show that 
\begin{equation}\label{eq:theta*-theta}
\theta_{z_*;r_*}^{(\lambda)} \leq
64^\frac{(q+1)(n+p^\sharp+q+1)}{2p^\sharp + n(1-q)}
\theta_{z;r}^{(\lambda)}.
\end{equation}
Observe that if $\theta_{z_*,r_*}^{(\lambda)} = \lambda$ (which implies $\tilde r_* = R_o$) we have
$$
\theta_{z_*;r_*}^{(\lambda)} = \lambda \leq \theta_{z;r}^{(\lambda)}.
$$
On the other hand, if $\lambda < \theta_{z_*;r_*}^{(\lambda)} (= \theta_{z_*;\tilde r_*}^{(\lambda)} = \tilde \theta_{z_*;\tilde r_*}^{(\lambda)})$ we have by~\eqref{eq:theta-tilde-intrinsic} that 
\begin{equation}\label{eq:theta*-intrinsic}
( \theta_{z_*;r_*}^{(\lambda)})^\frac{2 p^\sharp + n(1-q)}{q+1} =
\frac{\lambda^{p-2}}{|Q_{\tilde r_*}|} \iint_{Q_{\tilde
    r_*}^{(\lambda, \theta_{z_*;r_*}^{(\lambda)})} (z_*)}
\frac{|u|^{p^\sharp}}{\tilde r_*^{p^\sharp}} \, \d x \d t.
\end{equation}
Fix $\eta = 16$. By distinguishing between the cases $\tilde r_* \leq \frac{R_o}{\eta}$ and $\tilde r_* > \frac{R_o}{\eta}$, for the latter we obtain 
\begin{align*}
( \theta_{z_*;r_*}^{(\lambda)})^\frac{2 p^\sharp + n(1-q)}{q+1}
&\leq \lambda^{p-2} \left( \frac{R}{\tilde r_\ast} \right)^{n+p^\sharp+q+1}
\biint_{Q_{R}(z_o)} \frac{|u|^{p^\sharp}}{R^{p^\sharp}} \, \d x \d t\\
&\leq (4 \eta)^{n+p^\sharp+q+1} ( \theta_{z;r}^{(\lambda)})^\frac{2 p^\sharp + n(1-q)}{q+1}
\end{align*}
similarly as in~\eqref{eq:theta-R0-lambda}, since $\lambda \le \theta_{z;r}^{(\lambda)}$. For the former case, we may assume that $\theta_{z_*;r_*}^{(\lambda)} \geq \theta_{z;r}^{(\lambda)}$ since otherwise~\eqref{eq:theta*-theta} clearly holds. Furthermore, observe that $r \leq 2 r_* \leq 2 \tilde r_* \leq \eta \tilde r_*$, which implies
$$
\theta_{z_*;r_*}^{(\lambda)} \geq \theta_{z;r}^{(\lambda)} \geq \theta_{z;\eta \tilde r_*}^{(\lambda)}.
$$
Thus, we have
$$
B_{4 \tilde r_*}^{(\theta_{z_*,r_*}^{(\lambda)})} (x_*) \subset B_{\eta \tilde r_*}^{(\theta_{z, \eta \tilde r_*}^{(\lambda)})} (x).
$$
Using this together with~\eqref{eq:reverse-time-inclusion} to
estimate the right-hand side of~\eqref{eq:theta*-intrinsic} from
above, we deduce
\begin{align*}
	( \theta_{z_*;r_*}^{(\lambda)})^\frac{2 p^\sharp + n(1-q)}{q+1}
	&\leq
	\frac{ \eta^{p^\sharp}\lambda^{p-2}}{|Q_{\tilde r_*}|} \iint_{Q_{\eta \tilde r_*}^{(\lambda, \theta_{z,\eta \tilde r_*}^{(\lambda)})} (z)} \frac{|u|^{p^\sharp}}{ (\eta \tilde r_*)^{p^\sharp}} \, \d x \d t \\
	&\leq
	\eta^{n+p^\sharp+q+1} (\theta_{z;r}^{(\lambda)})^\frac{2 p^\sharp + n(1-q)}{q+1},
\end{align*}
where we used Lemma~\ref{lem:theta-rho-prop} (i) with
$\rho=s=\eta\tilde r_\ast$ for the last estimate.
Therefore, we have shown that~\eqref{eq:theta*-theta} holds in every
case. By choosing 
$$
\hat c \geq  4 \big(4\cdot 64^\frac{(q-1)(n+p^\sharp+q+1)}{2p^\sharp + n(1-q)}+1\big) \geq 20,
$$
it follows that $B_{4r}^{(\theta_{z;r})}(x) \subset B_{\hat c r_*}^{(\theta_{z_*;r_*})}(x_*)$. This is due to the fact that for every $x_1\in B_{4r}^{(\theta_{z;r})}(x)$ we have
\begin{align*}
|x_1 - x_*| &\leq |x_1 - x| + |x - x_*| \leq 2 \theta_{z;r}^\frac{1-q}{1+q} (4r) + \theta_{z_*;r_*}^\frac{1-q}{1+q} (4r_*) \\
&\leq 4 \theta_{z_*;r_*}^\frac{1-q}{1+q} r_* ( 4\cdot64^\frac{(q-1)(n+p^\sharp+q+1)}{2p^\sharp + n(1-q)}  + 1 ) \leq \hat c \theta_{z_*;r_*}^\frac{1-q}{1+q} r_*,
\end{align*}
where we used $Q \cap  Q^*\neq \varnothing$, $r < 2r_*$ and~\eqref{eq:theta*-theta}. By also recalling~\eqref{eq:time-inclusion}, we have 
$$
Q = Q_{4r}^{(\lambda,\theta_{z;r}^{(\lambda)})} (z) \subset \widehat Q^* = Q_{\hat c r_*}^{(\lambda,\theta_{z_*;r_*}^{(\lambda)})} (z_*),
$$
which completes the proof.
\end{proof}

\subsection{Stopping time argument}
\label{sec:stopping-time}

Let
\begin{equation} \label{eq:lambdao}
\lambda_o := 1+ \left[ \biint_{Q_{4R}} \frac{|u|^{p^\sharp}}{(4R)^{p^\sharp}} + |Du|^p + |F|^p \, \d x \d t  \right]^\frac{d}{p}.
\end{equation}
Consider $\lambda > \lambda_o$ and $r\in (0,2R]$ and define
$$
\mathbf E(r,\lambda) := \big\{ z \in Q_r: z \text{ is a Lebesgue point of } |Du| \text{ and } |Du|(z)> \lambda \big\},
$$
in which Lebesgue points are understood in context of cylinders of the type $Q_{\rho}^{(\lambda,\theta_{\rho})}$ constructed in Section~\ref{subsec:cylinder-construction}.

Consider radii $R \leq R_1 < R_2 \leq 2 R$ and concentric cylinders $Q_R \subset Q_{R_1} \subset Q_{R_2} \subset Q_{2R}$. Fix $z_o \in \mathbf E (R_1,\lambda)$ and denote $\theta_s = \theta^{(\lambda)}_{z_o; s}$ for $s \in (0,R_o]$. By definition of $\mathbf{E}(R_1,\lambda)$ we have
\begin{equation} \label{eq:Du-leb-point}
\liminf_{s\to 0}\biint_{Q_s^{(\lambda,\theta_s)}(z_o)} |Du|^p + |F|^p \, \d x \d t \geq |Du|^p(z_o) > \lambda^p.
\end{equation}
Let $\hat c$ denote the constant from the Vitali type covering lemma, Lemma~\ref{lem:vitali}, and consider
\begin{equation} \label{eq:lambda-B-lambdao}
\lambda > B \lambda_o, \quad \text{where}\quad B := \left( \frac{4 \hat c R}{R_2-R_1} \right)^\frac{d p^\sharp(n+2)(q+1)}{p (2p^\sharp + n(1-q))} >1.
\end{equation}
Let $\frac{R_2-R_1}{\mathfrak m} \leq s \leq R_o$, where $\mathfrak m
= \hat c \lambda^\frac{p^\sharp +1 -
  p-q}{q+1}$. By~\eqref{eq:lambdao}, 
Lemma~\ref{lem:theta-rho-prop}\,(iii) and \eqref{choice-R_o} we have
\begin{align*}
\biint_{Q_s^{(\lambda,\theta_s)}(z_o)} &|Du|^p + |F|^p \, \d x \d t \leq \frac{|Q_{4R}|}{\big|Q_s^{(\lambda,\theta_s)}\big|}  \biint_{Q_{4R}} |Du|^p + |F|^p \, \d x \d t \\
&\leq \left(\frac{4R}{s} \right)^{n+q+1} \lambda^{p-2} \theta_s^\frac{n(q-1)}{q+1} \lambda_o^\frac{p}{d} \\
&\leq \left(\frac{4R}{s} \right)^{n+q+1}  \left(\frac{4R_o}{s} \right)^\frac{n(q-1)(n+p^\sharp+q+1)}{2p^\sharp + n(1-q)} \lambda^{p-2+n\frac{q-1}{q+1}} \lambda_o^\frac{p}{d} \\
&\leq \lambda^\frac{(p^\sharp +1 - p-q)(n+q+1)}{q+1} \left( \frac{4\hat c R}{R_2 - R_1} \right)^\frac{p^\sharp(n+2)(q+1)}{2p^\sharp + n(1-q)} \lambda^{p-2+n\frac{q-1}{q+1}} \lambda_o^\frac{p}{d} \\
&= (B \lambda_o)^\frac{p}{d} \lambda^{p^\sharp-q-1 + n \frac{p^\sharp -p}{q+1}} < \lambda^p.
\end{align*}
By the above estimate,~\eqref{eq:Du-leb-point} and the continuity of the integral (w.r.t. $s$) there exists a maximal radius $\rho_{z_o} \in (0,\frac{R_2-R_1}{\mathfrak m})$ such that 
\begin{equation} \label{eq:lambda-rho}
\biint_{Q_{\rho_{z_o}}^{(\lambda,\theta_{\rho_{z_o}})}(z_o)} |Du|^p + |F|^p \, \d x \d t = \lambda^p.
\end{equation}
The maximality of the radius implies 
\begin{equation}\label{eq:lambda-s}
\biint_{Q_{s}^{(\lambda,\theta_{s})}(z_o)} |Du|^p + |F|^p \, \d x \d t < \lambda^p \quad \text{for every } s \in (\rho_{z_o}, R_o].
\end{equation}
By combining the last inequality with Lemma~\ref{lem:theta-rho-prop}\,(ii) and using the fact that $\rho \mapsto \theta_\rho$ is non-increasing, we have
\begin{align} \label{eq:grad-s-lambda}
\biint_{Q_{s}^{(\lambda,\theta_{\rho_{z_o}})}(z_o)} |Du|^p + |F|^p \, \d x \d t &\leq \left( \frac{\theta_{\rho_{z_o}}}{\theta_s} \right)^{n \frac{q-1}{q+1}}\biint_{Q_{s}^{(\lambda,\theta_{s})}(z_o)} |Du|^p + |F|^p \, \d x \d t \nonumber \\
&< \left( \frac{s}{\rho_{z_o}} \right)^\frac{n(q-1)(n+ p^\sharp +q+ 1)}{2 p^\sharp + n(1-q)} \lambda^p
\end{align}
for every $s \in (\rho_{z_o}, R_o]$. Observe that also clearly $Q_{\hat c \rho_{z_o}}^{(\lambda,\theta_{\rho_{z_o}})} (z_o) \subset Q_{R_2}$.

\subsection{A reverse H\"older inequality}
\label{sec:reverse-Holder-final-section}

Fix $z_o \in \mathbf E (R_1,\lambda)$ and $\lambda > B \lambda_o$ as defined in~\eqref{eq:lambda-B-lambdao}. We will show that
\begin{align} \label{eq:reverse-holder-rho}
\biint_{Q_{\rho_{z_o}}^{(\lambda,\theta_{\rho_{z_o}})}(z_o)} |Du|^p \, \d x \d t &\leq c \left( \biint_{Q_{4 \rho_{z_o}}^{(\lambda,\theta_{\rho_{z_o}})}(z_o)} |Du|^{\nu p} \, \d x \d t \right)^\frac{1}{\nu} \\
&\quad + c \biint_{Q_{4 \rho_{z_o}}^{(\lambda,\theta_{\rho_{z_o}})}(z_o)} |F|^p \, \d x \d t, \nonumber
\end{align}
for exponents $\max \left\{\frac{n(q+1)}{p(n+q+1)}, \frac{p-1}{p}, \frac{n}{n+2},\frac{n}{n+2} \left(1+\frac{2}{p}- \frac{2}{q} \right) \right\} \leq \nu \leq 1$ and a constant $c = c(n,p,q,C_o,C_1) > 0$.

First, we consider the case $\tilde \rho_{z_o} \leq 2
\rho_{z_o}$. Observe that this implies $\tilde \rho_{z_o}<R_o$, and
therefore $\lambda < \theta_{\rho_{z_o}} = \theta_{\tilde \rho_{z_o}}
= \tilde \theta_{\tilde \rho_{z_o}}$. By
Lemma~\ref{lem:theta-rho-prop}\,(i) with $s = 2 \tilde \rho_{z_o}$
and~\eqref{eq:theta-tilde-intrinsic} we have
$$
\biint_{Q_{2 \tilde \rho_{z_o}}^{(\lambda, \theta_{\rho_{z_o}})}(z_o)} \frac{|u|^{p^\sharp}}{(2 \tilde \rho_{z_o})^{p^\sharp}} \, \d x \d t \leq \theta_{\rho_{z_o}}^\frac{2p^\sharp}{q+1} = \biint_{Q_{\tilde \rho_{z_o}}^{(\lambda, \theta_{\rho_{z_o}})}(z_o)} \frac{|u|^{p^\sharp}}{\tilde \rho_{z_o}^{p^\sharp}} \, \d x \d t,
$$
i.e., condition~\eqref{eq:theta-intrinsic} holds with $C_\theta = 1$
and $\rho=\tilde\rho_{z_o}$. By~\eqref{eq:grad-s-lambda} and~\eqref{eq:lambda-rho} we deduce
\begin{align*}
4^\frac{n(1-q)(n+p^\sharp+q+1)}{2p^\sharp + n(1-q)} &\biint_{Q_{2 \tilde \rho_{z_o}}^{(\lambda, \theta_{\rho_{z_o}})}(z_o)} |Du|^p + |F|^p \, \d x \d t \\
&< \lambda^p = \biint_{Q_{\rho_{z_o}}^{(\lambda,\theta_{\rho_{z_o}})}(z_o)} |Du|^p + |F|^p \, \d x \d t \\
&\leq 2^{n+q+1} \biint_{Q_{\tilde \rho_{z_o}}^{(\lambda,\theta_{\rho_{z_o}})}(z_o)} |Du|^p + |F|^p \, \d x \d t,
\end{align*}
which implies that also~\eqref{eq:lambda-intrinsic} holds with $C_\lambda = C_\lambda (n,p,q)$. Thus, we can use Lemma~\ref{lem:reverse-holder} to obtain
\begin{align*}
\biint_{Q_{\rho_{z_o}}^{(\lambda,\theta_{\rho_{z_o}})}(z_o)} |Du|^p \, \d x \d t &\leq 2^{n+q+1} \biint_{Q_{\tilde \rho_{z_o}}^{(\lambda,\theta_{\rho_{z_o}})}(z_o)} |Du|^p \, \d x \d t \\
& \leq c \left( \biint_{Q_{4 \rho_{z_o}}^{(\lambda,\theta_{\rho_{z_o}})}(z_o)} |Du|^{\nu p} \, \d x \d t \right)^\frac{1}{\nu} \\
&\quad + c \biint_{Q_{4 \rho_{z_o}}^{(\lambda,\theta_{\rho_{z_o}})}(z_o)} |F|^p \, \d x \d t,
\end{align*}
for $c = c(n,p,q,C_o,C_1)$. This proves~\eqref{eq:reverse-holder-rho}
in the first case. 

Then, we consider the case $\tilde \rho_{z_o} > 2 \rho_{z_o}$. Observe that by~\eqref{eq:grad-s-lambda} and~\eqref{eq:lambda-rho} we have
\begin{align*}
2^\frac{n(1-q)(n+p^\sharp+q+1)}{2p^\sharp + n(1-q)} &\biint_{Q_{2 \rho_{z_o}}^{(\lambda, \theta_{\rho_{z_o}})}(z_o)} |Du|^p + |F|^p \, \d x \d t \\
&< \lambda^p 
= \biint_{Q_{\rho_{z_o}}^{(\lambda,\theta_{\rho_{z_o}})}(z_o)} |Du|^p + |F|^p \, \d x \d t,
\end{align*}
such that~\eqref{eq:lambda-intrinsic} holds with $C_\lambda =
C_\lambda(n,p,q)$ and
$\rho=\rho_{z_o}$. Furthermore,~\eqref{eq:theta-degenerate-1}$_1$ with
$C_\theta = 1$ holds by Lemma~\ref{lem:theta-rho-prop}\,(i). For the
proof of \eqref{eq:theta-degenerate-1}$_2$, we first consider the case
$\tilde \rho_{z_o} \in [\frac{R_o}{2},R_o]$. In this case, by Lemma~\ref{lem:theta-rho-prop}\,(iii) and~\eqref{eq:lambda-rho} we have
\begin{align*}
	\theta_{\rho_{z_o}}^p
	=
	\theta_{\tilde \rho_{z_o}}^p
	&\leq
	8^\frac{p(q+1)(n+ p^\sharp +q+ 1)}{2 p^\sharp + n(1-q)} \lambda^p \\
	&=
	8^\frac{p(q+1)(n+ p^\sharp +q+ 1)}{2 p^\sharp + n(1-q)} \biint_{Q_{\rho_{z_o}}^{(\lambda,\theta_{\rho_{z_o}})}(z_o)} |Du|^p + |F|^p \, \d x \d t,
\end{align*}
which implies~\eqref{eq:theta-degenerate-1}$_2$ with
$C_\lambda=C_\lambda(n,p,q)$.
Now we are left with the
case $\tilde \rho_{z_o} \in (2 \rho_{z_o},\frac{R_o}{2})$. Observe
that since $\tilde \rho_{z_o} < R_o$, it follows that $\lambda <
\theta_{\rho_{z_o}} = \theta_{\tilde \rho_{z_o}} = \tilde
\theta_{\tilde \rho_{z_o}}$ by definition so that
Lemma~\ref{lem:theta-rho-prop}\,(i) and~\eqref{eq:theta-tilde-intrinsic} imply
\begin{equation*}
\biint_{Q_{2 \tilde \rho_{z_o}}^{(\lambda, \theta_{\rho_{z_o}})}(z_o)} \frac{|u|^{p^\sharp}}{(2 \tilde \rho_{z_o})^{p^\sharp}} \, \d x \d t \leq \theta_{\rho_{z_o}}^{\frac{2p^\sharp}{q+1}} = \biint_{Q_{\tilde \rho_{z_o}}^{(\lambda, \theta_{\rho_{z_o}})}(z_o)} \frac{|u|^{p^\sharp}}{\tilde \rho_{z_o}^{p^\sharp}} \, \d x \d t.
\end{equation*}
Furthermore, by $\theta_{\rho_{z_o}} = \theta_{\tilde \rho_{z_o}}$,
the monotonicity of $\rho\mapsto\theta_\rho$, Lemma~\ref{lem:theta-rho-prop}\,(ii) and~\eqref{eq:lambda-s} we obtain
\begin{align*}
\biint_{Q_{2 \tilde \rho_{z_o}}^{(\lambda, \theta_{\rho_{z_o}})}(z_o)}
  |Du|^p + |F|^p \, \d x \d t
  &\le
  \bigg(\frac{\theta_{\tilde\rho_{z_o}}}{\theta_{2\tilde\rho_{z_o}}}\bigg)^{n\frac{q-1}{q+1}}\biint_{Q_{2 \tilde \rho_{z_o}}^{(\lambda, \theta_{2\tilde \rho_{z_o}})}(z_o)}
  |Du|^p + |F|^p \, \d x \d t\\
  &< 2^\frac{n(q-1)(n+p^\sharp+q+1)}{2p^\sharp + n(1-q)}\lambda^p.
\end{align*}
Thus, $Q_{\tilde \rho_{z_o}}^{(\lambda, \theta_{\rho_{z_o}})}(z_o)$ is $\theta$-intrinsic (with $C_\theta = 1$) and $\lambda$-subintrinsic. We use Lemmas~\ref{lem:theta-absorb} and~\ref{lem:theta-rho-prop}\,(i) (observe that $\tilde \rho_{z_o} / 2 > \rho_{z_o}$) to obtain 
$$
\theta_{\rho_{z_o}}^\frac{2}{q+1} \leq c \lambda^\frac{2}{q+1} + \frac{3}{4} \left( \biint_{Q_{\tilde \rho_{z_o}  /2}^{(\lambda,\theta_{\rho_{z_o}})}(z_o)} \frac{|u|^{p^\sharp}}{(\tilde \rho_{z_o} / 2)^{p^\sharp}} \, \d x \d t \right)^\frac{1}{p^\sharp} \leq c \lambda^\frac{2}{q+1} + \frac{3}{4} \theta_{\rho_{z_o}}^\frac{2}{q+1}.
$$
Thus, by~\eqref{eq:lambda-rho}
$$
\theta_{\rho_{z_o}} \leq c \lambda = c \biint_{Q_{\rho_{z_o}}^{(\lambda,\theta_{\rho_{z_o}})}(z_o)} |Du|^p + |F|^p \, \d x \d t
$$
holds true, which implies~\eqref{eq:theta-degenerate-1}$_2$ with $C_\theta=C_\theta(n,p,q,C_o,C_1)$ also in
this final case. Therefore, we have established
that~\eqref{eq:lambda-intrinsic} and \eqref{eq:theta-degenerate-1}
hold true with $\rho=\rho_{z_o}$ in the case $\tilde\rho_{z_o}>2\rho_{z_o}$. 
This enables us to use Lemma~\ref{lem:reverse-holder} to conclude
that~\eqref{eq:reverse-holder-rho} holds in any case. 

\subsection{Final argument}
\label{sec:final-argument}

The rest of the proof is identical to~\cite[Sect. 6.5 \& 6.6]{Boegelein-Duzaar-Scheven:2022}. Hence, we refrain ourselves from repeating the computations and only sketch the final argument. 

We have that if $\lambda$ satisfies~\eqref{eq:lambda-B-lambdao}, then for every $z_o\in \mathbf{E}(R_1,\lambda)$ there exists a cylinder $Q_{\rho_{z_o}}^{(\lambda,\theta_{z_o;\rho_{z_o}})} (z_o)$ in which~\eqref{eq:lambda-rho},~\eqref{eq:lambda-s},~\eqref{eq:grad-s-lambda} and~\eqref{eq:reverse-holder-rho} hold true and Lemma~\ref{lem:vitali} is satisfied. Furthermore, $Q_{\hat c\rho_{z_o}}^{(\lambda,\theta_{z_o;\rho_{z_o}})} (z_o) \subset Q_{R_2}$ in which $\hat c$ is the constant from Lemma~\ref{lem:vitali}.

By denoting
$$
\mathbf{F}(r,\lambda) := \left\{ z \in Q_{r} : z \text{ is a Lebesgue point of } |F| \text{ and } |F|(z) > \lambda \right\},
$$
we deduce as in~\cite[Sect. 6.5]{Boegelein-Duzaar-Scheven:2022} that
\begin{align*}
\iint_{\mathbf{E}(R_1,\tilde \lambda)} |Du|^p \, \d x \d t \leq c \iint_{\mathbf{E}(R_2,\tilde \lambda)} \tilde \lambda^{(1-\nu)p} |Du|^{\nu p} \, \d x \d t +c \iint_{\mathbf{F}(R_2,\tilde \lambda)} |F|^p \, \d x \d t
\end{align*}
for every $\tilde \lambda \geq \eta B \lambda_o$, in which $\eta = \eta (n,p,q,C_o,C_1) \in (0,1]$ and $B$ and $\lambda_o$ are defined in~\eqref{eq:lambda-B-lambdao} and~\eqref{eq:lambdao}. 

By a truncation and Fubini type argument, the estimate in Theorem~\ref{thm:main} can be deduced exactly as in~\cite[Sect. 6.6]{Boegelein-Duzaar-Scheven:2022}.


\begin{thebibliography}{10}

\bibitem{Alonso-etal}
\newblock R.~Alonso, M.~Santillana and C.~Dawson.
\newblock On the diffusive wave approximation of the shallow water equation.
\newblock \emph{Euro.~J.~Appl.~Math.} 19(5), 575--606 (2008).

\bibitem{Bamberger-etal}
\newblock A.~Bamberger, M.~Sorine and J.P.~Yvon.
\newblock \emph{Analyse et contr\^ole d'un r\'eseau de transport de gaz. (French)}
\newblock In: R.~Glowinski, J.L.~Lions (eds) Computing Methods in Applied Sciences and Engineering, II.~Lecture Notes in Physics, vol 91. Springer, Berlin, Heidelberg (1977).
% https://doi.org/10.1007/3-540-09119-X_110
  
\bibitem{Boegelein:1}
\newblock V.~B\"ogelein.
\newblock Higher integrability for weak solutions of higher order degenerate parabolic systems. 
\newblock {\em Ann.~Acad.~Sci.~Fenn.~Math.} 33(2), 387--412 (2008).

\bibitem{Boegelein-Duzaar}
\newblock V.~Bögelein and F.~Duzaar.
\newblock Higher integrability for parabolic systems with non-standard growth and degenerate diffusions.
\newblock \emph{Publ.~Mat.} 55(1), 201--250 (2011).

\bibitem{BDKS:2018}
\newblock V.~B\"ogelein, F.~Duzaar, J.~Kinnunen and C.~Scheven.
\newblock Higher integrability for doubly nonlinear parabolic systems.
\newblock \emph{J.~Math.~Pures Appl.} 143, 31--72 (2020). 

\bibitem{Boegelein-Duzaar-Korte-Scheven}
\newblock V.~B\"ogelein, F.~Duzaar, R.~Korte and C.~Scheven.
\newblock The higher integrability of weak solutions of porous medium systems.
\newblock {\em Adv.~Nonlinear Anal.} 8(1), 1004--1034 (2019).


\bibitem{Boegelein-Duzaar-Scheven:2022}
\newblock V.~B\"ogelein, F.~Duzaar and C.~Scheven.
\newblock Higher integrability for doubly nonlinear parabolic systems.
\newblock \emph{Partial Differ.~Equ.~Appl.} 3, 74 (2022).
% https://doi.org/10.1007/s42985-022-00204-0

\bibitem{Boegelein-Duzaar-Scheven:2018}
\newblock V.~B\"ogelein, F.~Duzaar and C.~Scheven.
\newblock Higher integrability for the singular porous medium system.
\newblock \emph{J.~Reine Angew.~Math.} 767, 203--230 (2020). 


\bibitem{Boegelein-Parviainen}
\newblock V.~B\"ogelein and M.~Parviainen.
\newblock Self-improving property of nonlinear higher order parabolic systems near the boundary. 
\newblock  {\em NoDEA Nonlinear Differential Equations Appl.} 17(1), 21--54 (2010).

\bibitem{DiBe}
E.~DiBenedetto.
\newblock {\em Degenerate parabolic equations}.
\newblock Universitext, Springer, New York (1993).
%https://doi.org/10.1007/978-1-4612-0895-2

\bibitem{DbF2}
E.~DiBenedetto and A.~Friedman.
\newblock {H\"older} estimates for non-linear degenerate parabolic systems.
\newblock {\em J.~Reine Angew.~Math.}, 357, 1--22 (1985).

\bibitem{DBGV-book} 
E.~DiBenedetto, U.~Gianazza and V.~Vespri. 
\newblock {\it Harnack's inequality for degenerate and singular parabolic equations}.
\newblock Springer Monographs in Mathematics, Springer, New York (2011).
%https://doi.org/10.1007/978-1-4614-1584-8

\bibitem{Gehring}
\newblock F.~W.~Gehring.
\newblock The $L^p$-integrability of the partial derivatives of a quasiconformal mapping.
\newblock {\em Acta Math.} 130, 265--277 (1973).

\bibitem{Gianazza-Schwarzacher}
\newblock U.~Gianazza and S.~Schwarzacher.
\newblock Self-improving property of degenerate parabolic equations of porous medium-type.
\newblock \emph{Amer.~J.~Math.} 141(2), 399--446 (2019). 

\bibitem{Gianazza-Schwarzacher-singular}
\newblock U.~Gianazza and S.~Schwarzacher.
\newblock Self-improving property of the fast diffusion equation.
\newblock \emph{J.~Funct.~Anal.} 277(12), 108291 (2019).

\bibitem{Giaquinta:book}
\newblock M.~Giaquinta.
\newblock \emph{Multiple integrals in the calculus of variations and nonlinear elliptic systems}.
\newblock Princeton University Press, Princeton (1983).

\bibitem{Giaquinta-Struwe}
\newblock M.~Giaquinta and M.~Struwe.
\newblock On the partial regularity of weak solutions of nonlinear parabolic systems.
\newblock {\em Math.~Z.} 179(4), 437--451 (1982).

\bibitem{Giusti:book}
\newblock E.~Giusti.
\newblock {\em Direct methods in the calculus of variations}.
\newblock World Scientific Publishing Company, Tuck Link, Singapore (2003).

\bibitem{Kalashnikov}
\newblock A.~S.~Kalashnikov.
\newblock Some problems of the qualitative theory of nonlinear degenerate second-order
parabolic equations.
\newblock \emph{Russian Math.~Surveys} 42(2), 169--222 (1987).
%https://doi.org/10.1070/RM1987v042n02ABEH001309

\bibitem{Kim-deg}
\newblock W.~Kim, J.~Kinnunen and K.~Moring.
\newblock Gradient higher integrability for degenerate parabolic double-phase systems.
\newblock \emph{Arch.~Ration.~Mech.~Anal.}, 247, 79 (2023).
%https://doi.org/10.1007/s00205-023-01918-0

\bibitem{Kim-sing}
\newblock W.~Kim and L.~Särkiö.
\newblock Gradient higher integrability for singular parabolic double-phase systems.
\newblock \emph{Preprint}, \href{http://arxiv.org/abs/2310.07386v1}{arXiv:2310.07386v1}.

\bibitem{Kinnunen-Lewis:1}
\newblock J.~Kinnunen and J.~L.~Lewis.
\newblock Higher integrability for parabolic systems of $p$-Laplacian type.
\newblock {\em Duke Math.~J.} 102(2), 253--271 (2000).

\bibitem{Kinnunen-Lewis:very-weak}
J.~Kinnunen and J.~L.~Lewis.
\newblock Very weak solutions of parabolic systems of p-Laplacian type.
\newblock \emph{Ark.~Mat.} 40(1), 105--132 (2002).

\bibitem{Leugering-Mophou}
\newblock G.~Leugering and G.~Mophou.
\newblock \emph{Instantaneous optimal control of friction dominated flow in a gas-network}.
\newblock In: V.~Schulz, D.~Seck (eds) Shape Optimization, Homogenization and Optimal Control.
International Series of Numerical Mathematics, vol 169. Birkhäuser, Cham (2018).

\bibitem{Mahaffy}
\newblock M.~W.~Mahaffy.
\newblock A three-dimensional numerical model of ice sheets: Tests on the Barnes ice cap, northwest territories.
\newblock \emph{J.~Geophys.~Res} 81(6), 1059--1066 (1976). 

\bibitem{Meyers-Elcrat}
\newblock N.~G.~Meyers and A.~Elcrat.
\newblock Some results on regularity for solutions of non-linear elliptic systems and quasi-regular functions.
\newblock {\em Duke Math.~J.} 42, 121--136 (1975).

\bibitem{Moring-pme-global}
\newblock K.~Moring, C.~Scheven, S.~Schwarzacher and T.~Singer.
\newblock Global higher integrability of weak solutions of porous medium systems.
\newblock \emph{Comm.~Pure Appl.~Math.} 19(3), 1697--1745 (2020).

\bibitem{Parviainen}
\newblock M.~Parviainen.
\newblock Global gradient estimates for degenerate parabolic equations in nonsmooth domains. 
\newblock{\em Ann.~Mat.~Pura Appl.} 188, 333--358 (2009).
%https://doi.org/10.1007/s10231-008-0079-0

\bibitem{Parviainen-singular}
\newblock M.~Parviainen.
\newblock Reverse H\"older inequalities for singular parabolic equations near the boundary.
\newblock \emph{J.~Differential Equations} 246(2), 512--540 (2009).

\bibitem{Saari-Schwarzacher}
\newblock O.~Saari and S.~Schwarzacher.
\newblock A reverse Hölder inequality for the gradient of solutions to Trudinger's equation.
\newblock \emph{NoDEA Nonlinear Differential Equations Appl.} 29, 24 (2022).

\bibitem{Urbano}
\newblock J.~M.~Urbano.
\newblock \emph{The method of intrinsic scaling. A systematic approach to regularity for degenerate and
singular PDEs.}
\newblock Lecture Notes in Mathematics, Springer, Berlin (2008).

\bibitem{Vazquez}
\newblock J.~L.~V\'azquez.
\newblock \emph{Smoothing and decay estimates for nonlinear diffusion equations. Equations of porous medium type}.
\newblock Oxford Lecture Series in Mathematics and its Applications, Oxford University Press, Oxford (2006).


\end{thebibliography}
\end{document}